\documentclass[3pt,preprint]{elsarticle}
\usepackage{threeparttable}
\usepackage{lineno,hyperref}
\usepackage{verbatim}
\usepackage{amsmath}
\usepackage{CJK}
\usepackage{algorithm}
\usepackage{algorithmicx}
\usepackage{algpseudocode}
\usepackage{color}
\usepackage{xcolor}
\modulolinenumbers[5]
\usepackage{mathrsfs}
\usepackage{amssymb}
\usepackage{subfigure} 
\DeclareMathAlphabet{\mathscr}{OT1}{pzc}{m}{it}
\usepackage{booktabs}
\usepackage{graphicx}
\usepackage{ntheorem}
\usepackage{multirow}
\usepackage{pifont}
\usepackage{enumitem}
\usepackage{ulem}
\usepackage[numbers]{natbib}
\usepackage{listings}
\journal{Pattern Recognition}
\newtheorem{remark}{Remark}
\newtheorem{theorem}{Theorem}
\newtheorem{lemma}{Lemma}
\newtheorem*{proof}{Proof}

\newtheorem{property}{Property}
\newtheorem{proposition}{Proposition}
\newcommand{\expo}{\overset{\circ}{\operatorname{exp}}}

\begin{document}

\begin{frontmatter}

\title{CSGO: Constrained-Softassign Gradient Optimization For Large Graph Matching\tnoteref{mytitlenote}}

\tnotetext[mytitlenote]{The research is supported by Natural Science Foundation of China (12271047); Guangdong Provincial Key Laboratory of Interdisciplinary Research and Application for Data Science, BNU-HKBU United International College (2022B1212010006); UIC research grant (UICR0400036-21C, UICR0400008-21; R04202405-21); Guangdong College Enhancement and Innovation Program (2021ZDZX1046).}

\author[bnubj,bnuzh]{Binrui Shen}
\ead{binrui.shen19@bnu.edu.cn}
\author[XJTLU]{Qiang Niu}
\ead{qiang.niu@xjtlu.edu.cn}
\author[bnu,uic]{Shengxin Zhu}
\ead{Shengxin.Zhu@bnu.edu.cn}


\address[bnubj]{School of Mathematical Sciences, Laboratory of Mathematics and Complex Systems, MOE, Beijing Normal University, 100875 Beijing, P.R.China}

\address[bnuzh]{Faculty of Arts and Sciences, Beijing Normal University, 519087 Zhuhai, P.R.China}

\address[bnu]{Advanced Institute of Natural Science, Beijing Normal University, Zhuhai 519087, P.R.China}

\address[uic]{Guangdong Provincial Key Laboratory of Interdisciplinary Research and Application for Data Science, BNU-HKBU United International College, Zhuhai 519087, P.R.China}





\begin{abstract}
Graph matching aims to find correspondences between two graphs. This paper integrates several well-known graph matching algorithms into a framework: the constrained gradient method. The primary difference among these algorithms lies in tuning a step size parameter and constraining operators. By leveraging these insights, we propose an adaptive step size parameter to guarantee the underlying algorithms' convergence, simultaneously enhancing their efficiency and robustness. For the constraining operator, we introduce a scalable softassign for large graph matching problems. Compared to the original softassign, our approach offers increased speed, improved robustness, and reduced risk of overflow. The advanced constraining operator enables a CSGO for large graph matching, which outperforms state-of-the-art methods in experiments. Notably, in attributed graph matching tasks, CSGO achieves an over 10X increase in speed compared to current constrained gradient algorithms.



\end{abstract}


\end{frontmatter}


\section{Introduction}
Graph matching aims to find correspondences between graphs with potential relationships. It can be used in various fields of intelligent information processing, e.g., activity analysis \cite{chen2012efficient},
shape matching \cite{2011Scale, 2010Learning}, detection of similar pictures \cite{shen2020fabricated}, graph similarity computation \cite{lan2022aednet,lan2022more}, medical image \cite{mh2024lvm}, knowledge graph alignment \cite{xu2019cross}, autonomous driving \cite{song2023graphalign}, alignment of vision-language models \cite{nguyen2024logramedlongcontextmultigraph}, and COVID-19 disease mechanism study \cite{gordon2020comparative}.

The graph matching problem is normally formulated as a quadratic assignment problem (QAP) that is NP-hard \cite{garey1979computers}. Despite decades of research on graph matching, it is still challenging to acquire a global optimum solution in acceptable computational time. Relaxation is a common technique to obtain an approximate solution at an acceptable cost. Popular approaches include but are not limited to the spectral-based method \cite{caelli2004inexact,  hermanns2023grasp, robles2007riemannian,umeyama1988eigendecomposition}, probabilistic modeling method \cite{egozi2012probabilistic}, random walk method \cite{2010Reweighted}, path-following method \cite{maron2018probably, wang2017graph, zaslavskiy2008path} and optimal transport method \cite{xu2019scalable, xu2019gromov}.

Among recently proposed graph matching algorithms, continuous optimization algorithms have attracted much research attention due to their outstanding performance \cite{ 2010Reweighted,gold1996graduated,leordeanu2005spectral, leordeanu2009integer}. Continuous optimization algorithms first solve a relaxed form of the NP-hard matching problem and obtain a discrete solution by converting the continuous solution back to the discrete domain. Recent advances in the fundamental connections between discrete graph theory and continuous problems have been established in \cite{chang2021lovasz}. Popular continuous optimization algorithms include spectral matching (SM) \cite{leordeanu2005spectral}, graduated assignment (GA) \cite{gold1996graduated}, integer projected fixed-point method (IPFP) \cite{leordeanu2009integer}, and doubly stochastic projected fixed-point method (DSPFP) \cite{lu2016fast}. These algorithms can be encapsulated within a constrained gradient method which updates the solution by applying the constrained gradient. Their differences lie in constraining operators and selecting the step size parameter, as summarized in Table \ref{tab:algorithms}. These algorithms solve graph matching problems, formulated by the Lawler quadratic assignment problem, with $O(n^4)$ cost. The cost of these algorithms limits their application in large graph matching problems \cite{lu2016fast}. To address this issue, \citet{lu2016fast} shows that constrained gradient methods enjoy only $O(n^3)$ cost by modeling the graph matching problem as the Koopmans-Beckmann quadratic assignment problem (KB-QAP). Constrained gradient algorithms, such as GA, SM, and IPFP, can enjoy this scalability in KB-QAP. Although DSPFP \cite{lu2016fast} achieves superior performance in large graph matching problems, two challenges persist: (a) the constrained gradient method \cite{lu2016fast} for KB-QAP lacks an efficient strategy to tune the step size parameter, as the strategy proposed in \cite{leordeanu2009integer} has a time complexity of $O(n^4)$; (b) existing constrained operators still lack robustness or speed in large graph matching, which is discussed in Section 3 and summarized in Table \ref{tab:operators}.




%


\begin{table}[ht]
\centering

\label{tab:algorithms}
\resizebox{0.6\columnwidth}{!}{%
\begin{tabular}{cccc}
\hline
Algorithms                                  & Relaxation        & Constraining operator & Step size parameter \\ \hline
GA\cite{gold1996graduated} & doubly stochastic & softassign    & 1                   \\
SM\cite{leordeanu2005spectral}  & spectral          & norm-normalization & 1        \\
IPFP\cite{leordeanu2009integer} & doubly stochastic & discrete projection                        & adaptive \\
DSPFP\cite{lu2016fast}          & doubly stochastic & alternating projection               & 0.5      \\ \hline
\end{tabular}%
}
\caption{Summary of constrained gradient algorithms}
\end{table}

Our main contributions are summarized as follows.
\begin{enumerate}[label=(\arabic*), leftmargin=1.5cm] 

\item  \textbf{Scalable Softassign}: we propose a scalable softassign method as a constraining operator for large graph matching problems with four key improvements: a) demonstrating the sensitivity of softassign to nodes' number $n$ and providing a solution by tuning a parameter; b) normalizing the magnitude of the input matrix to ensure the stability of the algorithm; c) mitigating the risk of overflow by the log-sum-exp technique; d) employing a vectorized iteration to double the efficiency. Compared with another SOTA constraining operator \cite{lu2016fast},  scalable softassign exhibits superior convergence properties.

 \item \textbf{Step size parameter tuning}:  we propose an adaptive step size parameter for the constrained gradient algorithms solving the KB-QAP. It guarantees the convergence of algorithms and enhances their efficiency and robustness. Compared with the existing strategy in \cite{leordeanu2009integer}, the computational cost is reduced from $O(n^4)$ to $O(n^3)$ in each iteration. Further, we simplify the computation of the optimal step size parameter, thereby enhancing the efficiency of constrained gradient algorithms in addressing both the Lawler QAP and the KB-QAP.

  \item \textbf{Maching algorithm:} 
  Scalable Softassign yields a constrained-softassign gradient optimization (CSGO) for large graph matching problems. A proposed warm-start strategy further enhances the speed.  In addition, we prove that CSGO is invariant to a scaling and shifting change on the affinity matrix and feature matrix of the matching graphs (other constrained gradient algorithms also enjoy this property). CSGO outperforms other constrained gradient methods in terms of both accuracy and speed by an order of magnitude.

\begin{figure}
    \centering
    \includegraphics[width=0.45\linewidth]{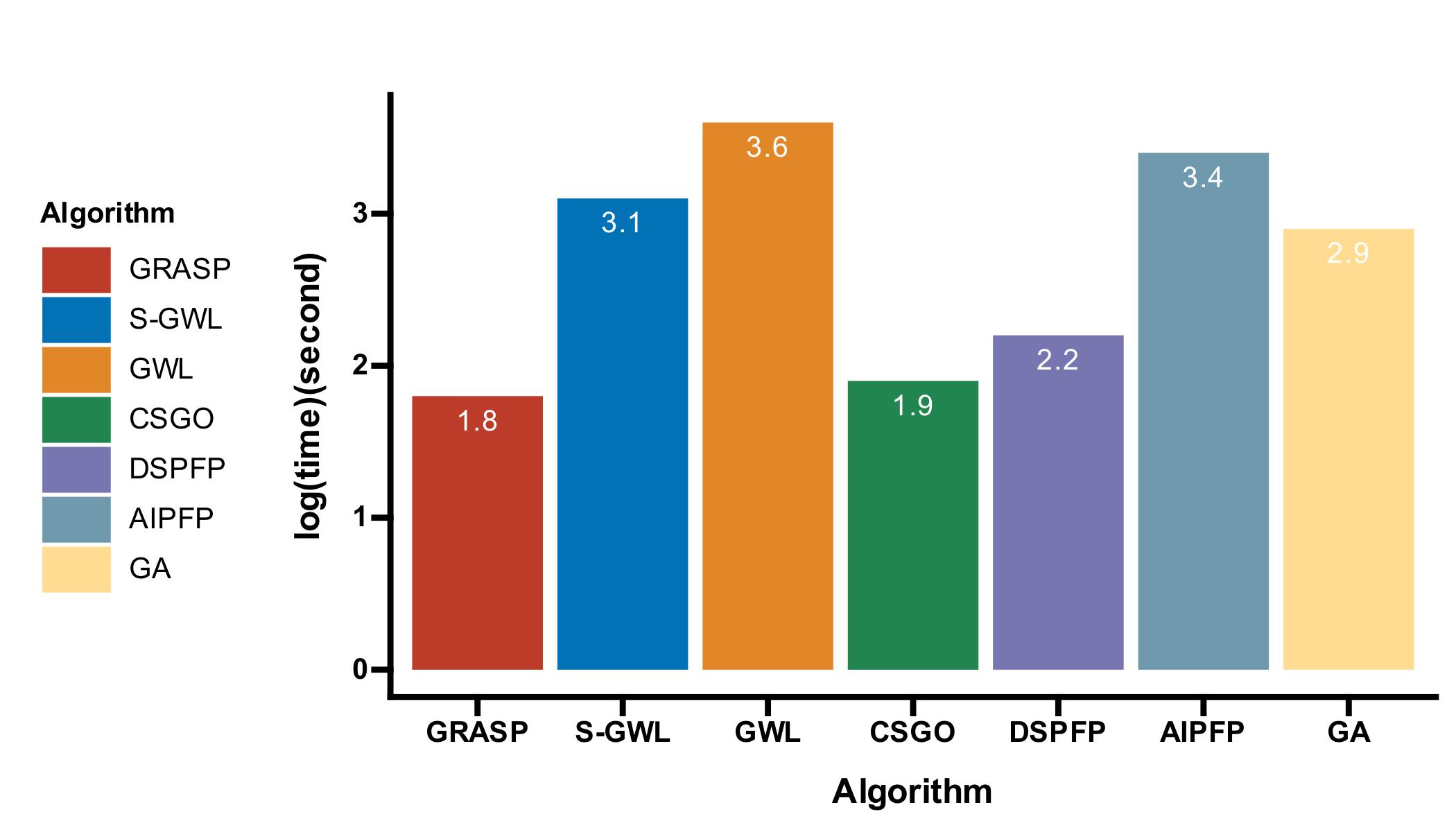}\includegraphics[width=0.45\linewidth]{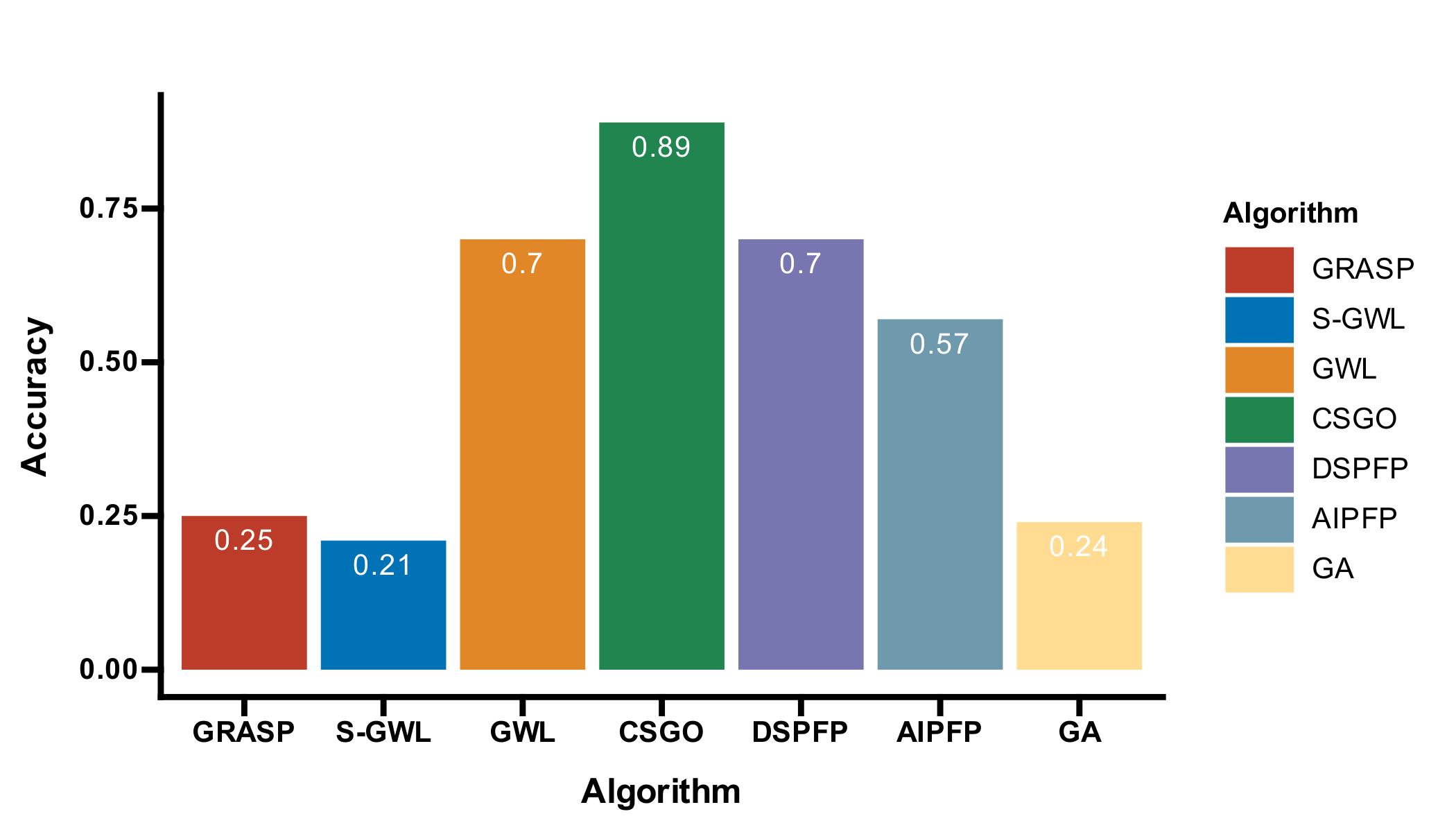}
        
    \caption{Mean matching accuracy and running time of different algorithms on Facebook networks.}
    \label{fig:performance}
\end{figure}
     \end{enumerate}

The rest of this paper is organized as follows. Section 2 reviews the formulation of the graph matching problem and the constrained gradient method. Section 3 proposes a scalable softassign. Section 4 discusses the selection of the step size parameter, which is a crucial parameter of the constrained gradient method. Section 5 introduces a warm-start strategy, the matching algorithm CSGO, and robustness analysis. Section 6 investigates the effects of the adaptive step size parameter and the warm-start strategy on graph matching performance. In this section, the comparison between CSGO and other graph matching methods is also performed. Section 7 gives a summary and discusses future work.

\section{Preliminaries}
This section introduces the attributed graph, the matching matrix for graph correspondences, the formulation of graph matching problems, and a constrained gradient method for solving them. Common symbols are summarised in Table \ref{tab:symbols}.
\begin{table}[]
\centering
\caption{Symbols and Notations.}
\resizebox{0.45\columnwidth}{!}{%
\begin{tabular}{cc}
\hline
Symbol                       & Definition                         \\ \hline
${G},\tilde{{G}}$                 & matching graphs                   \\
$A,\tilde{A}$                          & edge attribute matrices of  ${G}$ and $\tilde{{G}}$                \\
$F,\tilde{F}$                          & node attribute matrices of  ${G}$ and $\tilde{{G}}$          \\
$n,\tilde{n}$                          & number of nodes of  ${G}$ and $\tilde{{G}}$                  \\
$M$                          & matching matrix                   \\
$\Pi_{n \times n}$           & set of $n \times n$ permutation matrices       \\
$\Sigma_{n \times n}$        & set of $n \times n$ doubly stochastic matrices \\\hline
$\mathbf{1}$,$ \mathbf{0}$   & a column vector of all 1s,0s      \\ 
${D}_{( {\mathbf{x} })}$ & diagonal matrix of a vector $\mathbf{x}$                      \\
$\operatorname{tr}(\cdot)$                  & trace                             \\
$\operatorname{vec}(\cdot)$   & the vectorization of a matrix\\

$\langle \cdot,\cdot\rangle$ & inner product                     \\
$\|\cdot \|_{Fro}$                    &Frobenius norm                 \\
$\expo$                       & element-wise exponential                \\

$\oslash$                    & element-wise division                                   \\ \hline
$\beta$                      & the parameter in softassign       \\
$\alpha$                      & the step size parameter       \\
$\mathcal{P}_{sk}(\cdot)$    & Sinkhorn method                   \\
 \hline
\end{tabular}%
}
\end{table}
\label{tab:symbols}
\subsection{Attributed graph and matching matrix}

An \textit{undirected attributed graph} $G=\{V,E,A,F\}$ consists of a finite set of nodes $V = \{1, . . . ,n\}$ and a set of edges $E \subset V \times V$. $A$ is a nonnegative symmetric \textit{affinity matrix} whose element $A_{ij}$ represents an attribute of $E_{ij}$. The $i_{th}$ row of feature matrix $F$ represents the attribute vector of $V_{i}$. If the nodes' attributes are constants, $A_{ii}$ can represent the attribute of $V_{i}$. In addition, if nodes and edges do not have attributes, the graph is a plain graph, and $A$ is an adjacency matrix.

Given two attributed graphs $G=\{V,E,A,F\}$ and $\tilde{G}=\{\tilde{V},\tilde{E},\tilde{A},\tilde{F}\}$, we first consider the same cardinality of the vertice $n = \tilde{n}$ for simplicity. A matching matrix $M\in \mathbb{R}^{n\times n}$ can encodes the correspondence between nodes:
\begin{equation}
    M_{i \tilde{i}}=\left\{\begin{array}{ll}
1 & \text { if node } i \text { in } G \text { corresponds to node } \tilde{i} \text { in } \widetilde{G}, \\
0 & \text { otherwise. }
\end{array}\right.
\end{equation}
Subject to the one-to-one constraint, a matching matrix is a permutation matrix. The set of permutation matrices is denoted as $\Pi_{n \times n} = \{M\mathbf{1}=\mathbf{1},{M}^{T}\mathbf{1}=\mathbf{1}, {M}\in\{0,1\}^{n\times n}\}$ where $\mathbf{1}$ represents vectors with all-one. 




\subsection{Graph matching problem and relaxation}
In \cite{lu2016fast, zaslavskiy2008path}, the graph matching problem is formulated as a Koopmans-Beckmann QAP
\begin{equation}
\min\limits_{M\in \Pi_{n \times n}} \underbrace{\frac{1}{2}
\left\|A-M\widetilde{A}M^{T}\right\|_{Fro}^{2}}_{\text{Edges' dissimilarites}}+\underbrace{\lambda \left\|F-M \widetilde{F}\right\|_{Fro}^{2}}_{\text{Noes' dissimilarites}},
\label{eq. object 1}
\end{equation}
where $\lambda$ is a parameter and $\| \cdot \|_{Fro}$ is the Frobenius norm.  As noticed $\|X\|_{Fro}^2={\operatorname{tr}\left(XX^T\right)}$, the problem  \eqref{eq. object 1} is equivalent to 
\begin{equation}
\max\limits_{M\in \Pi_{n \times n}}  \frac{1}{2} \operatorname{tr}\left({M}^{T} {A} {M} {\widetilde{A}}\right)+\lambda \operatorname{tr}\left({M}^{T}{K}\right),  
\label{eq.Object_fast}
\end{equation}
where $K=F\tilde{F}^T$ and $\operatorname{tr(\cdot)}$ is the trace operator.

Due to the discrete constraints, \eqref{eq.Object_fast}  is an NP-hard problem \cite{1963The}. \textit{Relaxation} is a common technique to address this problem: firstly, a relaxed form of the discrete problem is solved, which typically yields an easier-to-obtain solution; secondly, the solution is converted back into the initial discrete domain by solving a linear assignment problem, as discussed in Section 3. Our algorithm aims to optimize the following continuous problem, in which we only drop the integer constraints from the problem \eqref{eq.Object_fast}:

\noindent \textbf{Problem}
\begin{equation}
\max\limits_{M\in \Sigma_{n \times n}}\mathcal{Z}(M),   \ \mathcal{Z}(M) =\frac{1}{2} \operatorname{tr}\left({M}^{T} {A} {M} {\widetilde{A}}\right)+\lambda \operatorname{tr}\left({M}^{T}{K}\right),  
\label{eq.Object_relaxedKB}
\end{equation}
where $\Sigma_{n \times n}:=\{M: M\mathbf{1}=\mathbf{1}, M^{T}\mathbf{1}=\mathbf{1}, M \geq 0$\} is the convex hull of the set of permutation matrices.



\subsection{Constrained gradient method}
The constrained gradient method \cite{lu2016fast} can find a relaxed solution of problem \eqref{eq.Object_fast}:
\begin{equation}
\begin{array}{cc}

    M^{(t+1)}=(1- \alpha)M^{(t)}  + \alpha D^{(t)} ,\\
    D^{(t)} =  \mathcal{P}(\nabla \mathcal{Z}(M^{(t)}))  = \mathcal{P}(AM^{(t)}\tilde{A}+\lambda K),
    \end{array}
    \label{iter.DSPFP}
\end{equation}
where $\alpha$ is a step size parameter and $\mathcal{P}(\cdot)$ is a constraining operator used to obtain a constrained gradient matrix satisfying the underlying constraints. With this formulation, spectral matching, graduated assignment method, doubly stochastic projected fixed-point method, and integer fixed-point method can enjoy $O(n^3)$ cost and $O(n^2)$ space complexity (derivation details are shown in the Appendix). These algorithms differ in the step size parameter and constraining operators, which affects their performance significantly. 





There are four common constraining operators. The simplest constraining operator, norm-normalization,  only concerns the norm and nonnegativity of the gradient matrix. Three tighter constraining operators are the Hungarian method \cite{kuhn1955hungarian} used in IPFP, the greedy linear method \cite{leordeanu2005spectral} used in approximating integer projected fixed-point method (AIPFP) \cite{lu2016fast}, and the alternating projection method \cite{zass2006doubly} used in DSPFP \cite{lu2016fast}.  Four constraining operators are summarized in Table \ref{tab:operators}, and more details are shown in Appendix. Enlightened by these works, we develop a scalable softassign method for large graph matching problems.

\begin{table}[]
\centering
\resizebox{\textwidth}{!}{%
\begin{tabular}{cccccc}
Types      & Constraint sets        & \multicolumn{1}{c}{Constraining operators} & \multicolumn{1}{c}{Complexity} &\multicolumn{1}{c}{Limitations} & \multicolumn{1}{c}{Algorithms}\\ \hline
\multicolumn{1}{l}{} & \multicolumn{1}{l}{} &                                &                                \\
\begin{tabular}[c]{@{}c@{}}spectral \end{tabular} &
  \begin{tabular}[c]{@{}c@{}}$\|P\| =1$\\ $P \geq 0$\end{tabular} &
  \multicolumn{1}{c}{\begin{tabular}[c]{@{}c@{}}norm-normalization\\  operator \end{tabular}}&
  \multicolumn{1}{c}{\begin{tabular}[c]{@{}c@{}}$O(n^2)$\end{tabular}} &\multicolumn{1}{c}{sensity to noises} &\multicolumn{1}{c}{SM} \\
\multicolumn{1}{l}{} & \multicolumn{1}{l}{} &                                 &                                \\ \hline
\multirow{3}{*}{\begin{tabular}[c]{@{}c@{}} \\ doubly \\ stochastic\end{tabular}} &
  \multirow{3}{*}{\begin{tabular}[c]{@{}c@{}} $P \in \Sigma_{n \times n}$\end{tabular}} &
  \multicolumn{1}{c}{ \textbf{softassign}} &
  \multicolumn{1}{c}{$O(n^2)$ each iteration} & \multicolumn{1}{c}{sensity to nodes' cardinality} & \multicolumn{1}{c}{GA}\\
                     &                      &                                 &                                \\
 &
   & \multicolumn{1}{c}{\begin{tabular}[c]{@{}c@{}}{alternating} \\ {projection}\end{tabular}} & $O(n^2)$ each iteration&
  \multicolumn{1}{c}{\begin{tabular}[c]{@{}c@{}} poor convergence \end{tabular}} &
  \multicolumn{1}{c}{\begin{tabular}[c]{@{}c@{}}DSPFP \end{tabular}}\\ \hline
\multicolumn{1}{c}{} & \multicolumn{1}{c}{} &                                 &                                \\
\begin{tabular}[c]{@{}c@{}}permutation \end{tabular} &
  \begin{tabular}[c]{@{}c@{}}$P \in \Pi_{n \times n}$\end{tabular} &
  \multicolumn{1}{c}{\begin{tabular}[c]{@{}c@{}}Hungarian  method \\  \\ greedy method \end{tabular}} &
  \multicolumn{1}{c}{\begin{tabular}[c]{@{}c@{}} $O(n^3)$ \\  \\  $O(n^3)$\end{tabular}} &
  \multicolumn{1}{c}{\begin{tabular}[c]{@{}c@{}} lose information\\ when multipe-solution exisits\end{tabular}} &
  \multicolumn{1}{c}{\begin{tabular}[c]{@{}c@{}} IPFP \\ \\ AIPFP \end{tabular}} \\
\multicolumn{1}{c}{} & \multicolumn{1}{c}{} &                                 &                                \\ \hline
\end{tabular}%

}
\caption{Commonly used constraining operators.}
\label{tab:operators}
\end{table}


\section{Scalable softassign method}
This section introduces the softassign method and the vectorized update, explains how to tune a parameter of softassign to balance efficiency and accuracy in large graph matching, and proposes a stable strategy to prevent numerical overflow.

\subsection{Softassign and vectorized update}
A natural choice for addressing the doubly stochastic constraints in problem \eqref{eq.Object_relaxedKB} is to use the doubly stochastic projection as a constraining operator in \eqref{iter.DSPFP}. The doubly stochastic projection finds the nearest doubly stochastic matrix to a given matrix X (the gradient matrix in \eqref{iter.DSPFP}) by solving
\begin{equation}
\mathcal{P}_{D}(X) = \arg \min_{P\in \Sigma_{n \times n}} \|P-X\|_{Fro}.
\label{eq:doubly assign}
\end{equation}
Since $\|X\|_{Fro}^2 = \operatorname{tr}(XX^T)$, it is easy to rewrite \eqref{eq:doubly assign} as a linear assignment problem
\begin{equation}
\mathcal{P}_{D}(X)=\arg \max_{P\in \Sigma_{n \times n}} \ \ \langle P, X \rangle,
\label{eq:linear assignment}
\end{equation}
where $\langle P, X \rangle = \sum P_{ij}X_{ij}$. However, the doubly stochastic projection \cite{lu2016fast} suffers from poor convergence when the numerical value of the input matrix is large \cite{rontsis2020optimal}. 

Softassign, proposed in \cite{kosowsky1994invisible}, can efficiently approximate the solution of \eqref{eq:linear assignment} by solving 
 \begin{equation}
    \mathcal{P}^{\beta}_S(X) = \arg \max_{P\in \Sigma_{n \times n}} \langle P, X \rangle + \frac{1}{\beta}\mathcal{H}(P), \ \ \ \mathcal{H}(P) = -\sum P_{ij} \ln {P_{ij}}.
    \label{eq: entropic assignment}
\end{equation}
where $\beta$ is a parameter (it represents the inverse of temperature in the physical energy model \cite{kosowsky1994invisible}). The advantage of the softassign method lies in its ability to balance efficiency and accuracy through the parameter $\beta$ \cite{cuturi2013Sinkhorn}, which is detailed in the next subsection. Softassign $\mathcal{P}^{\beta}_S(X)$ works as follows.

\begin{align}
\label{eq.oriSoftmax}
   S^{(0)}&=\expo(\beta X),\\
   S^{(t)} &=\mathcal{T}_{c}\left(\mathcal{T}_{r}\left(S^{(t-1)}\right)\right),
\label{eq.oriSoftmax0}
\end{align}
where 
\begin{align}
    \mathcal{T}_{r}(X)_{ij}&=X_{ij}/(\sum_j X_{ij}),\label{eq.oriSoftmax3}
    \\ 
    \mathcal{T}_{c}(X)_{ij}&=X_{ij}/(\sum_i X_{ij}).
    \label{eq.oriSoftmax2}
\end{align}
\eqref{eq.oriSoftmax3} and \eqref{eq.oriSoftmax2} are the normalization operators in rows and columns of a matrix. The exponential step $\expo(X)_{ij} = e^{X_{ij}}$ in \eqref{eq.oriSoftmax} ensures that all elements are positive and increases the differences between the elements. The iterative process in \eqref{eq.oriSoftmax0} is the Sinkhorn method \cite{Sinkhorn1967concerning}.

We adopt a fast implementation of Sinkhorn method \cite{cuturi2013Sinkhorn}:

\begin{equation}
    \begin{array}{cc}
    \mathbf{r}^{(t+1)}  =  {\mathbf{1}} \oslash ({S^{(0)}  \mathbf{c}^{(t)}}),
    \\
    \mathbf{c}^{(t+1)}  =  {\mathbf{1}} \oslash ({(S^{(0)})^{\mathrm{T}}  \mathbf{r}^{(t+1)}}),
    \end{array}
        \label{eq:softassign_sinkhorn}
\end{equation}
where $\oslash$ denotes the element-wise division. The relation between two versions of the Sinkhorn method is that $S^{(t)} = \operatorname{diag}({\mathbf{r}^{(t)}}) S^{(0)} \operatorname{diag}({\mathbf{c}^{(t)}})$ \cite{cuturi2013Sinkhorn}. This fast formula \eqref{eq:softassign_sinkhorn} only updates two $1 \times n$ vectors ${\mathbf{r}^{(t)}}$ and ${\mathbf{c}^{(t)}}$ instead of a $n \times n$ matrix $S^{(t)}$, which is calculated after convergence. Compared to original iterations \eqref{eq.oriSoftmax0}, the vectorized iteration \eqref{eq:softassign_sinkhorn} requires only half of the cost for each iteration.

\subsection{Robustness through parameter tuning}
The selection of $\beta$ determines the efficiency and accuracy of softassign. As $\beta$ increases, $\mathcal{P}^{\beta}_S(X)$ approaches the solution of the doubly stochastic projection \eqref{eq:doubly assign}, which yields a better performance. $\mathcal{P}^{\infty}_S(X)$ is the optimal solution of \eqref{eq:doubly assign}. However, a larger $\beta$ may increase iterations in the Sinkhorn method \cite{cuturi2013Sinkhorn}.

In large graph matching problems, the selection of $\beta$ must consider two key factors: the numerical magnitude of the input matrix and the number of nodes. These two factors have not received much attention in previous works, such as the reweighted random walks matching \cite{2010Reweighted} and the graduated assignment method \cite{gold1996graduated}. In their studies, elements of the input matrix range from 0 to 1; nodes' numbers are less than 200. The limited range of two factors leads to the neglect of their impacts. However, in large graph matching problems, the dimension and magnitude of the input matrix may vary widely. Therefore, we aim to develop a scalable softassign that is invariant to these two factors.

We give an example to demonstrate the effect of the numerical magnitude on softassign with a fixed $\beta$:
$$X_1=\begin{pmatrix}
1 & 1.1\\
1.1 & 1
\end{pmatrix} \Rightarrow \mathcal{P}^{1}_S(X_1)=\begin{pmatrix}
0.48 & 0.52\\
0.52 & 0.48
\end{pmatrix},
$$

$$ X_2=\begin{pmatrix}
20 & 22\\
22 & 20
\end{pmatrix} \Rightarrow \mathcal{P}^{1}_S(X_2)=\begin{pmatrix}
0.12 & 0.88\\
0.88 & 0.12
\end{pmatrix}.
$$
In this example, $X_1$ and $X_2$ have the same direction, but different magnitudes cause different results:  $\mathcal{P}^{1}_S(X_1)$ approaches a matrix whose elements are equal and $\mathcal{P}^{1}_S(X_2)$ approaches the permutation matrix (corresponding to the optimal solution of the assignment problem \eqref{eq:doubly assign}). To eliminate the impact, we normalize the input matrix by its maximum before the exponential operator: $\hat{X} = \frac{X}{\max(X)}$. 


Increasing the dimension of the input matrix may decrease the performance of the softassign with a fixed $\beta$. We illustrate this phenomenon as follows.
\begin{proposition}
Let $ X \in \mathbb{R}^{n \times n}$, $\beta \in \mathbb{R}$, and $Y=\mathcal{P}^{\beta}_{S}(X)$. Then
\begin{equation}
    \lim_{n \to \infty} Y_{ij}=Y_{ab} \quad \text{  for  } \quad i,j,a,b =1, 2, \dots ,n.
\end{equation}
\end{proposition}
\begin{proof}
Consider the \(i_{th}\) row of the matrix \( X \), denoted by \( X_i = [x_1, x_2, \dots, x_n] \in \mathbb{R}^n \). After applying a softassign with only a row normalization, the corresponding row \( \hat{Y}_i = [y_1, y_2, \dots, y_n] \) is 
\begin{equation}
   y_j = \frac{e^{\beta x_j}}{s_n}, \quad \text{where} \quad s_n = \sum_{k=1}^{n} e^{\beta x_k}, \quad \beta \in \mathbb{R}, \quad j = 1, 2, \dots, n. 
\end{equation}
As the number of elements \( n \) increases, the normalization factor \( s_n \) dominates the individual exponential terms, leading to increasingly smaller differences between any two normalized components \( y_i \) and \( y_j \). Therefore, we conclude 
\begin{equation}
\lim_{n \to \infty} |y_i - y_j| = 0 \quad \text{for} \quad i,j = 1, 2, \dots, n.
\end{equation}
Similarly, we can deduce that other rows of \( X \) exhibit the same property. 

Since the sums of rows are normalized to 1, all elements of \( \hat{Y} \) converge to the same value as \( n \to \infty \). Additionally, further row or column normalization in the Sinkhorn process will not alter this result. Thus, as \( n \to \infty \), all components of \( Y_{ij} \) converge to the same value, which completes the proof.
\end{proof}
This proposition indicates that the differences between elements of $\mathcal{P}^{\beta}_{S}(X)$ with a fixed $\beta$ may approach zero in large graph matching problems. Such a $\mathcal{P}^{\beta}_{S}(X)$ is far from the optimal solution $\mathcal{P}^{\infty}_{S}(X)$. 



The softassign algorithm can be robust to changes in the number of nodes $n$ when the distance between $\mathcal{P}^{\beta}_{S}(X)$ and the optimal solution $\mathcal{P}^{\infty}_{S}(X)$ remains invariant with respect to $n$. To achieve that, we first define a metric to assess the distance between $\mathcal{P}^{\beta}_{S}(X)$ and the optimal solution $\mathcal{P}^{\infty}_{S}(X)$ for a given profit matrix $X \in \mathbb{R}^{n \times n}$. The total assignment score of an assignment matrix $\mathcal{P}^{\beta}_{S}(X)$ is $\langle \mathcal{P}^{\beta}_{S}(X), X \rangle$. A total assignment error can be defined as $\langle \mathcal{P}^{\infty}_S(X), X \rangle - \langle \mathcal{P}^{\beta}_S(X), X \rangle$ in optimal transport problems \cite{altschuler2017near}. Since the total assignment error is the sum of individual assignment errors, it is influenced by both the per assignment error and the number of dimensions $n$. To avoid the influence from $n$, we introduce an \textit{average assignment error}
\begin{equation}
    \epsilon^{\beta}_{X} = \frac{1}{n} \langle \mathcal{P}^{\infty}_S(X), X \rangle - \langle \mathcal{P}^{\beta}_S(X),X \rangle
\end{equation}
to quantify the distance between $\mathcal{P}^{\beta}_S(X)$ and the optimal solution $\mathcal{P}^{\infty}_S(X)$.
\begin{proposition}
    Let $\gamma \in \mathbb{R}_+$, $X \in \mathbb{R}^{n \times n}$, the average assignment error for $\mathcal{P}^{\beta}_S(X)$
    \begin{equation}
   \epsilon^{\beta}_{X} \leq\frac{1}{\gamma},
\end{equation}
where $\beta = \gamma\ln(n)$.
\end{proposition}
\begin{proof}

\begin{equation}
    \mathcal{P}^{\beta}_S(X) = \arg \max_{P \in \Sigma_{n\times n}} \ \  \langle P, X \rangle + \frac{1}{\beta}H(P),
\end{equation}
 we have 
\begin{equation}
    \langle \mathcal{P}^{\infty}_S(X), X\rangle+0 \leq  \langle \mathcal{P}^{\beta}_S(X), X \rangle + \frac{1}{\beta}H(\mathcal{P}^{\beta}_S(X)) .
\end{equation}
For $P \in \Sigma_{n \times n}$, the entropy $\mathcal{H}(P)$ reaches its maximum when elements of $P$ are equal to $\frac{1}{n}$. Then we have
\begin{equation}
\mathcal{H}(\mathcal{P}^{\beta}_S(X)) \leq  n^2 (-\frac{1}{n}   \ln(\frac{1}{n}))=n\ln(n).
\end{equation}
Then 
\begin{equation}
    \langle \mathcal{P}^{\infty}_S(X), X\rangle - \langle \mathcal{P}^{\beta}_S(X), X \rangle \leq n(\frac{\ln(n)}{\beta}) =\frac{n}{\gamma},
\end{equation}
which completes the proof.
\end{proof}
This proposition suggests that the performance of softassign with $\beta = \gamma\ln(n)$ is invariant to the nodes' number. The choice of $\gamma$ depends on the type of the matching graphs. An empirical finding indicates that a challenging task requires a large $\gamma$ to ensure a low error bound. For instance, plain graph matching problems normally require a larger $\gamma$ compared to attributed graph matching problems.

\subsection{Stable strategy}
Numerical overflow of softassign is an inevitable issue in large graph matching problems. Using $\hat{X}$ can reduce but not eliminate the risk of overflow. To eliminate it, we modify the step in \eqref{eq.oriSoftmax} by a preprocessing 
\begin{equation}
S^{(0)} = \expo(\beta (\hat{X} - \mathbf{1}_{n \times n})),\ \hat{X} = \frac{X}{\max(X)}.
\label{eq.mod.exp}
\end{equation}
In this stage, the maximum of $S^{(0)}$ is 1, thereby avoiding the risk of overflow. We prove that this preprocessing will not affect the results in the following proposition.
\begin{proposition}
     Let $\beta \in \mathbb{R}$, $X   \in \mathbb{R}^{n \times n}$,  $B 
    \in \mathbb{R}^{n \times n}$ and its all elements are $b \in \mathbb{R}$, then 
\begin{equation}
        \mathcal{P}^{\beta}_S({X}-B) = \mathcal{P}^{\beta}_S({X}).
\end{equation}
\end{proposition}
\begin{proof}
The difference from the $B$ is eliminated in the first iteration of Sinkhorn
\begin{equation}
    \mathcal{T}_{r}(\expo(\beta (X - B)))=\mathcal{T}_{r}(\frac{\expo(\beta X)}{\exp(b)})\\
=\mathcal{T}_{r}(\expo(\beta X)),
\end{equation}
which completes the proof.
\end{proof}
This proposition is an application of the log-sum-exp technique whose rounding error analyzes are studied in \cite{blanchard2021accurately}. An interesting finding is that this strategy also improves final accuracy in certain graph matching tasks. For example, it increased accuracy by approximately 1.5\% in the protein-protein interaction networks matching experiment.


\subsection{Algorithm of scalable softassign}

Based on the improvements, a scalable softassign method is summarized in Algorithm \ref{ag.scalSoftassign}, and the softassign is shown in Algorithm \ref{ag.softassign}.

\begin{minipage}[t]{0.52\textwidth}
\vspace{0pt}
\begin{CJK*}{UTF8}{gkai}
    \begin{algorithm}[H]
        \caption{Scalable Softassign} 
        \begin{algorithmic}[1]
            \Require {$X,\gamma$}
                \State $\hat{X} = \frac{X}{\max(X)}$
                \State $\beta={\gamma\ln{n}}$ 
                \State $\hat{S}=\expo(\beta (\hat{X}  - \mathbf{1}_{n \times n} ))$ 
                \State $\mathbf{r}^{(0)},\mathbf{c}^{(0)} =\mathbf{1}$
                \While{$\mathbf{r}$ and $\mathbf{c}$ does not converge} 
                \State$\mathbf{r}^{(t)}={\mathbf{1}} \oslash ({\hat{S} \mathbf{c}^{(t-1)}})$, $\mathbf{c}^{(t)}={\mathbf{1}} \oslash ({\hat{S}^{T}  \mathbf{r}^{(t)}})$
                \EndWhile
                \State $S= \operatorname{diag}({\mathbf{r}^{(t)}}) \hat{S} \operatorname{diag}({\mathbf{c}^{(t)}})$
            \State \Return{$S$}
        \end{algorithmic}
        \label{ag.scalSoftassign}
    \end{algorithm}
\end{CJK*}
\end{minipage}
\hfill
\begin{minipage}[t]{0.41\textwidth}
\vspace{0pt}
\begin{CJK*}{UTF8}{gkai}
    \begin{algorithm}[H]
        \caption{Softassign} 
        \begin{algorithmic}[1]
            \Require {$X,\beta$}
                \State $S^{(0)}=\expo(\beta X)$ 
                \While{$S$ does not converge}
                \State $S^{(t)} =\mathcal{T}_{c}\left(\mathcal{T}_{r}\left(S^{(t-1)}\right)\right)$
                \EndWhile
            \State \Return{$S$}
        \end{algorithmic}
        \label{ag.softassign}
    \end{algorithm}
\end{CJK*}
\end{minipage}

The Sinkhorn method, Step 5-7, stops when $\|S^{(t)}\mathbf{1}-\mathbf{1}\|_1+\|(S^{(t)})^T\mathbf{1}-\mathbf{1}\|_1 \leq \varepsilon$. This Sinkhorn requires $O\left(\left(\varepsilon\right)^{-2} \log (s / \ell)\right)$ iterations, where $s=\sum_{i j} \hat{S}_{i j} \text { and } \ell=\min _{i j} \hat{S}_{i j}$ \cite[Th.2]{altschuler2017near}. In experiments with approximately 1000 nodes, the number of iterations typically does not exceed 30 for attributed graph matching. 


\section{Analysis of step-size parameter}
This section introduces an adaptive strategy for the step-size parameter of \eqref{iter.DSPFP}. We simplify the computation of the optimal $\alpha$ for algorithms with doubly stochastic constraining operators.
\subsection{Adaptive optimal step-size parameter}
Recall that the iterative formula of constrained gradient method \eqref{iter.DSPFP}  
\begin{equation}
    M^{(t+1)}=(1-\alpha)M^{(t)}+\alpha D^{(t)}.
\end{equation}
 Given $M^{(t)}$ and $D^{(t)}$, we can find the optimal $\alpha$ by a line search such that
\begin{equation}
    \begin{aligned}
(\alpha^*)^{(t)} = \arg \max_{\alpha \in[0,1]}  \quad  \mathcal{Z}((1-\alpha)M^{(t)}+\alpha D^{(t)}),
\end{aligned}
\label{eq.alpha}
\end{equation}
where $\mathcal{Z}$ is the objective function defined in \eqref{eq.Object_fast}. For simplify, the superscript index $(t)$ is ignored in the following derivation: $M$ and $D$ are used to replace $M^{(t)}$ and $D^{(t)}$ respectively; $\nabla \mathcal{Z}$ is used to replace $\nabla \mathcal{Z}(M^{(t)})$.

\begin{equation}
    \begin{aligned}\mathcal{Z}(M^{(t+1)}) &= \mathcal{Z}((1-\alpha)M^{(t)}+\alpha D^{(t)}) 
\\
&=\frac{1}{2} \operatorname{tr}\left(\left((1-\alpha) M^{T} A+\alpha D^{T} A\right)((1-\alpha) M \widetilde{A}+\alpha D\tilde{A}))\right.
\\
&+\lambda \operatorname{tr}\left(\alpha({D^{T} F \tilde{F}^{T}}-{M^{T} F \tilde{F}^{T}})+M^{T} F \tilde{F}^{T}\right)
\\
&=\frac{1}{2} \operatorname{tr}\left((1-\alpha)^{2} \underbrace{M^{T} A M \widetilde{A}}_{M_1}+(1-\alpha) \alpha \underbrace{M^{T} A D \tilde{A}}_{M_2}) \right.+ \alpha (1-\alpha) \underbrace{D^{T} A M\tilde{A}}_{M_3}
\\
& \left. +\alpha^{2} \underbrace{D^{T} A D \widetilde{A}}_{M_4}\right)+\lambda \operatorname{tr}\left(\alpha(\underbrace{D^{T} F \tilde{F}^{T}}_{M_{5}}-\underbrace{M^{T} F \tilde{F}^{T}}_{M_{6}})+M^{T} F \tilde{F}^{T}\right)
\\
&= \alpha^{2} (\underbrace{\frac{1}{2}\operatorname{tr}\left(M_{1}-M_{2}-M_{3}+M_{4}\right)}_{a_{\alpha}})+\alpha\underbrace{\operatorname{tr}(- M_{1}+M_{2}+\lambda (M_5-M_6)}_{b_{\alpha}})
\\
& + \underbrace{\operatorname{tr}\left(M_{1}+M_6\right)}_{c_{\alpha}}
\\
&=a_{\alpha} \alpha^{2}+ b_{\alpha}\alpha+c_{\alpha}
.\end{aligned}
\label{eq.alpha_general}
\end{equation}
The optimal step size parameter of each iteration can be obtained by maximiing such a quadratic problem. 


Equipped with the optimal step-size parameter in \eqref{eq.alpha}, constrained gradient algorithms have the following properties:
\begin{property}
The objective score $\mathcal{Z}(M^{(t)})$ is increasing at every step $t$.
\end{property}
\begin{proof}
Because the adaptive strategy of $\alpha$ guarantees that $\mathcal{Z}(M^{(t)}) \geq \mathcal{Z}(M^{(t-1)})$, we have
$\mathcal{Z}(M^{(t)}) \geq \mathcal{Z}(M^{(t-1)}) \geq \mathcal{Z}(M^{(t-2)}) \dots \mathcal{Z}(M^{(0)})$.
\end{proof}

\begin{property}
Constrained gradient algorithms with adaptive $\alpha$ converge to a maximum of the problem \eqref{eq.Object_relaxedKB}.
\end{property}
\begin{proof} If the sequence $(\mathcal{Z}(M^{(t)})$ is monotonic and bounded, then it is convergent. The monotony of the sequence is shown in property 1. The bound of the sequence also exists because $\mathcal{Z}(M) < \mathcal{Z}(\mathbf{1}_{n \times n})$ \text{for} where $\mathbf{1}_{n \times n}$ is a $n\times n$ matrix whose elements are 1. 
\end{proof}

\begin{property}
If all eigenvalues of two affinity matrices $A$ and $\tilde{A}$ are simultaneously positive (or negative), the optimal step-size parameter is 1.
\end{property}
\begin{proof} Firstly, if all eigenvalues of two affinity matrices $A$ and $\tilde{A}$ are positive (or negative), then the eigenvalues of $A \otimes \tilde{A}$ are also positive, i.e., $A \otimes \tilde{A}$  is a positive definite matrix. It follows that $a_{\alpha} = (\mathbf{m}-\mathbf{d})^TA \otimes \tilde{A} (\mathbf{m}-\mathbf{d}) \geq 0$ where $\mathbf{m} = \operatorname{vec}(M)$ and $\mathbf{d} = \operatorname{vec}(D)$, which means the optimal step-size parameter is 1.
\end{proof}


\begin{remark}
Two parameters $a_{\alpha}$ and $b_{\alpha}$ can also be obtained through
$a_{\alpha} = (\mathbf{m}-\mathbf{d})^T(A\otimes\tilde{A})(\mathbf{m}-\mathbf{d})$ and $b_{\alpha} = \mathbf{m}^T(A\otimes\tilde{A})(\mathbf{d}-\mathbf{m})$ \cite{leordeanu2009integer}. However, the computational cost of computing $\alpha$ based on this vector form is $O(n^4)$.
\end{remark}

\subsection{Optimal step-size parameter for doubly stochastic constraining operators}
The line search of $\alpha \in [0,1]$ can be simplified if the $D$ in \eqref{eq.alpha_general} results from a doubly stochastic constraining operator. Such a $D$ is an exact solution or an approximating solution (softassign) of the linear assignment problem 
\begin{equation}
D = \underset{X \in \Sigma_{n \times n}}{\arg \max}  \operatorname{tr}\left(X^{T}\nabla \mathcal{Z}\right).
\end{equation}
This implicits that $b_{\alpha} = \operatorname{tr}(D^{T}\nabla \mathcal{Z}-M^{T}\nabla\mathcal{Z})>0$ in \eqref{eq.alpha_general}. When $a_{\alpha}$ in \eqref{eq.alpha_general} is positive, the quadratic function \eqref{eq.alpha_general} is concave up, shown in Figure \ref{Fig.alpha_a} (a). Therefore, when $a_{\alpha}$ is positive, we can obtain $\alpha^*=1$ without the value of $b_{\alpha}$ . When $a_{\alpha}$ is negative, there are two cases shown in Figure \ref{Fig.alpha_a} (b).

\begin{figure}[h]
\centering  
    \subfigure[Positive $a_{\alpha}$]{
        \includegraphics[width=0.4\textwidth]{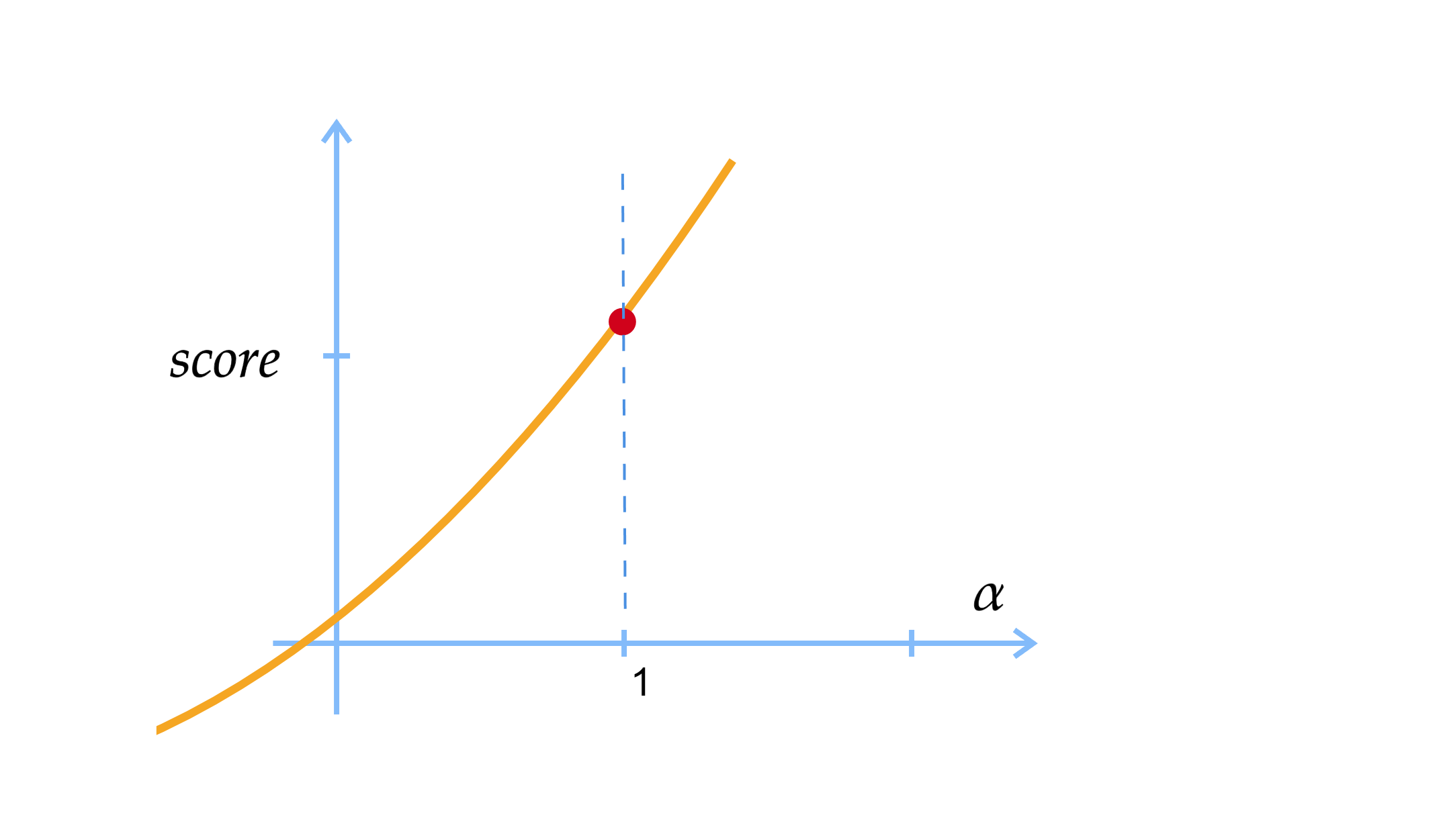}
    }
        \subfigure[Negative $a_{\alpha}$]{
        \includegraphics[width=0.55\textwidth]{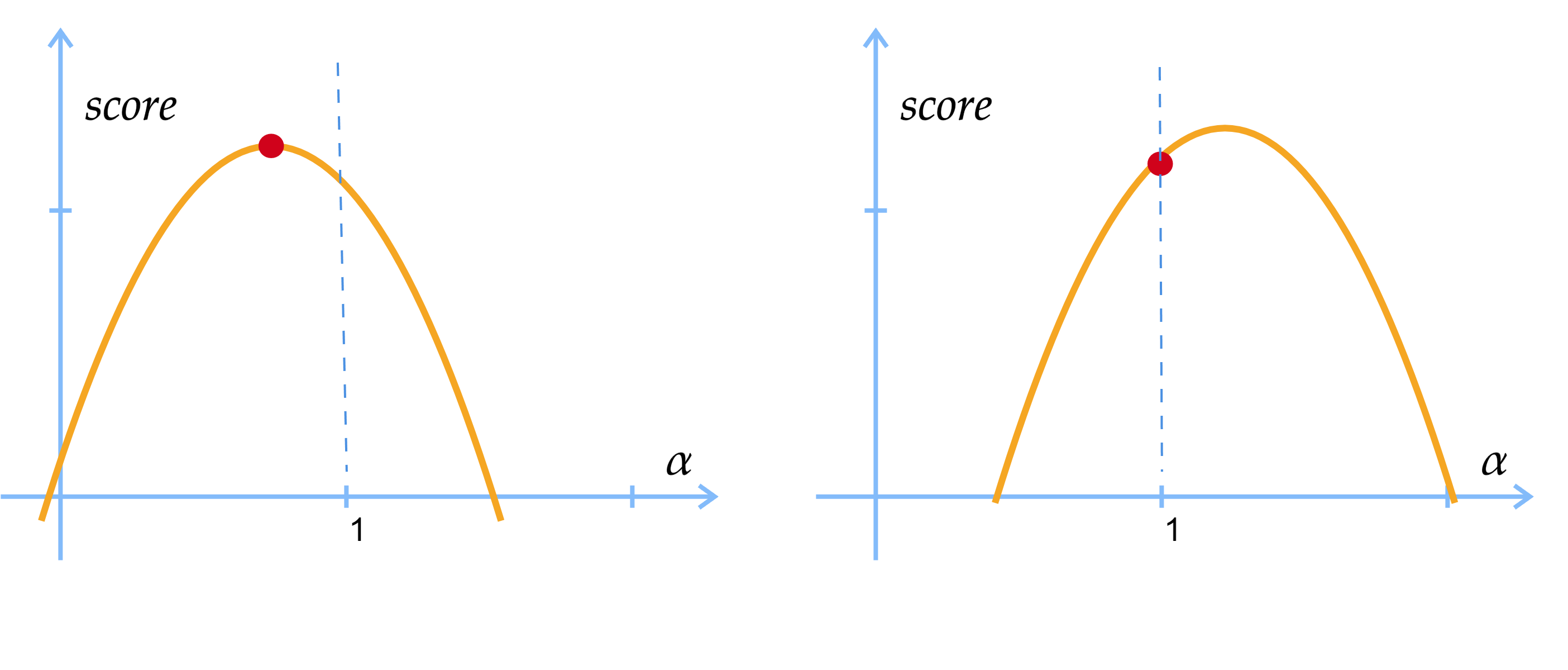}
    }

\caption{The relation between $\alpha$ and objective score. The red point represents the score with optimal $\alpha$.}
\label{Fig.alpha_a}
\end{figure}


In summary, if a doubly stochastic constraining operator is adapted in \eqref{iter.DSPFP}, the optimal step-size parameter can be simplified as follows
\begin{equation}
     \alpha^{(t)}=\left\{\begin{array}{cc}
1 & a_{\alpha}^{(t)} \geqslant 0, \\
\min \left(-\frac{b_{\alpha}^{(t)}}{2 a^{(t)}_{\alpha}}, 1\right) & a_{\alpha}^{(t)}<0.
\end{array}\right.
\label{eq:alpha_doubly}
\end{equation}
$b_{\alpha}^{(t)}$ is not necessary to compute when $a_{\alpha}^{(t)}$ is positive. This finding simplifies the computation of the optimal step-size parameter. Experiments show that $a_{\alpha}^{(t)}$ is positive in most cases, which further enhances the importance of this finding. The convenient strategy can be extended to constrained gradient methods for Lawler QAP.

\section{Constrained-Softassign Gradient Optimization Algorithm}
In this section, we propose a warm-start strategy for constrained gradient algorithms, introduce CSGO for large graph matching problems, and analyze its robustness.

\subsection{Warm-start strategy}
In the absence of prior matching information, the variable is typically initialized as a matrix where each element is set to $\frac{1}{n}$ i.e., $M^{(0)}= \frac{1}{n}\mathbf{1}\otimes \mathbf{1}^T$, indicating that the matching probability between any two nodes is equal \cite{leordeanu2009integer, lu2016fast}. We propose a warm-start approach under this initialization that enables rapid computation of the constraint gradient matrix during the first iteration of constrained gradient algorithms:
\begin{align}
    \nabla \mathcal{Z}(M^{(0)}) &= A(\frac{1}{n}\mathbf{1}\otimes \mathbf{1}^T)\tilde{A}+ K
    \\
    &=\frac{1}{n}(A\mathbf{1}) \otimes (\mathbf{1}^T\tilde{A})+ K
    \\
    &=\frac{1}{n}(A\mathbf{1}) \otimes (\tilde{A}\mathbf{1})^T+ K.
\end{align}
This strategy reduces the cost from $\mathcal{O}(n^3)$ to $\mathcal{O}(n^2)$. It notably enhances the efficiency of constrained gradient algorithms in attribute graph matching tasks, as algorithms typically converge within 3-5 iterations.


\subsection{Algorithm}
The CSGO is shown in Algorithm \ref{ag.CSGO}. Without loss of generality, we assume that $n<\tilde{n}$. For the problem of matching graphs with different nodes' numbers, \citet{gold1996graduated} introduce a slack square matrix like ${D}$ in step 3: ${D}_{(1:n,1:\tilde{n})} =A{N} \tilde{A}+\lambda {K}$ and the remaining elements of ${D}$ are zero. The updates of $N$ only requires ${D}_{(1:n,1:\tilde{n})}$ in step 6.

\begin{CJK*}{UTF8}{gkai}
    \begin{algorithm}[H]
        \caption{Sofsassign constrained gradient method}
        \begin{algorithmic}[1] 
            \Require $A,\tilde{A},K,\lambda,\gamma$
            \Ensure $M$       
                 \State Initial $D = \mathbf{0}_{n \times \tilde{n}}$
                 \While{$N$ does not converges} 
                 \If{interation = 1}{\\ \quad \quad\ \ ${D_{(1:n,1:\tilde{n})} = \frac{1}{n}(A\mathbf{1}) \otimes (\tilde{A}\mathbf{1})^T+ K}$}
                  \Else{\\ \quad \quad\ \ $D_{(1:n,1:\tilde{n})} =AN\tilde{A}+\lambda K$}
                   \EndIf

                    \State $D= \text{Scalable Softassign}(D,\gamma)$ by Algorithm 1
                    \State Compute the optimal $\alpha$ by \eqref{eq:alpha_doubly}
                    \State $N = (1-\alpha) N + \alpha D_{(1:n,1:\tilde{n})}$
                \EndWhile
                \State $Discretize$ $N$ $to$ $obtain$ $M$
                \State \Return{$M$}
        
        \end{algorithmic}
        \label{ag.CSGO}
    \end{algorithm}
\end{CJK*}

For $n =\tilde{n}$, Step 4 requires $O(n^2)$ operations.
Step 6 and Step 9 require $O(n^3)$ operations per iteration regardless of fast and sparse matrix computation. Step 8 requires $O({Tn^2})$ operations where $T$ is the number of iterations in the Sinkhorn method (discussed in the previous part). Step 12 transforms the doubly stochastic matrix $N$ back to a matching matrix $M$ using the Hungary method, which requires $O(n^3)$ operations in the worst-case scenario. In conclusion, this algorithm has time complexity $O(n^3+Tn^2)$ per iteration and space complexity $O(n^2)$.

\subsection{Robustness analysis}
This subsection analyzes the robustness of CSGO in response to scaling changes and offsets applied to the affinity matrix or feature matrix. The findings are encapsulated in the Theorem 1, whose proof is based on Lemmas 1 and 2.

\begin{theorem}
The matching formula 
\begin{equation}
\min_{M \in \Pi_{n \times n}} \frac{1}{2}\left\|A-M \widetilde{A} M^{T}\right\|_{F}^{2}+\lambda\|F-M \widetilde{F}\|_{F}^{2} 
\end{equation}
is invariant to a scaling change and offset on the adjacency matrix and feature matrices of graphs.
\end{theorem}

\begin{lemma}
For the edge part of the matching formula
\begin{equation}
\max_{{M \in \Pi_{n \times n}}}\quad \frac{1}{2} \operatorname{tr}\left({M}^{T} {A} {M} {\widetilde{A}}\right)
\label{eq.lemma.edge.ori}
\end{equation}
is invariant to a scaling change and offset on $A$. In other words, the problem \eqref{eq.lemma.edge.ori} is equivalent to 
\begin{equation}
\max_{{M \in \Pi_{n \times n}}}\quad \frac{1}{2} \operatorname{tr}\left({M}^{T} {A} {M} {(u\widetilde{A}+{q}\mathbf{1}\mathbf{1}^T)}\right)
\label{eq.lemma.edge.var}
\end{equation}
where $u$ is an arbitrary positive constant, and $q$   is an arbitrary constant.
\label{lemma.edge}
\end{lemma}

\begin{proof}
$$
\begin{aligned}
\operatorname{tr}\left({M}^{T} {A} {M} {(u\widetilde{A}+{q}\mathbf{1}\mathbf{1}^T)}\right)&=\operatorname{tr}\left( u{M}^{T} {A} {M} \widetilde{A} +{M}^{T} {A} {M}({q}\mathbf{1}\mathbf{1}^T) \right)
\\
&=\operatorname{tr}\left( u{M}^{T} {A} {M} \widetilde{A}\right) + \operatorname{tr}\left( {M}^{T} {A} {M}({q}\mathbf{1}\mathbf{1}^T) \right)
\\
&=u\operatorname{tr}\left( {M}^{T} {A} {M} \widetilde{A}\right) + \operatorname{tr}\left( {A} ({q}\mathbf{1}\mathbf{1}^T) \right).
\end{aligned}
$$
Because $\operatorname{tr}\left( {A} ({q}\mathbf{1}\mathbf{1}^T) \right)$ is a constant  and $u$ is a positive number, it is easy to deduce that two problems are equivalent.
\end{proof}

\begin{lemma}
The node term of the matching formula
\begin{equation}
\max_{{M \in \Pi_{n \times n}}}\quad\operatorname{tr}\left({M}^{T}{F \tilde{F}}\right)
\label{eq.lemma.node.ori}
\end{equation}
has an invariant property concerning the offset of $F$.  In other words, such a problem is equivalent to 
\begin{equation}
\max_{{M \in \Pi_{n \times n}}}\quad\operatorname{tr}\left({M}^{T}{F (u\tilde{F}+q\mathbf{1}\mathbf{1}^T)}\right),
\label{eq.lemma.node.var}
\end{equation}
where $u$ is an arbitrary positive constant, and $q$   is an arbitrary constant.
\label{lemma.node}
\end{lemma}

\begin{proof}
$$
\begin{aligned}
\operatorname{tr}\left({M}^{T}{F (u\tilde{F}+q)}\right) &= \operatorname{tr}\left({M}^{T}F (u\tilde{F}) +  {M}^{T}F (q\mathbf{1}\mathbf{1}^T)\right)
\\
&=\operatorname{tr}\left(u{M}^{T}F \tilde{F}\right) + \operatorname{tr}\left({M}^{T}F (q\mathbf{1}\mathbf{1}^T)\right)
\\
&=u\operatorname{tr}\left({M}^{T}F \tilde{F}\right) + \operatorname{tr}\left(F (q\mathbf{1}\mathbf{1}^T)\right).
\end{aligned}
$$
Because $\operatorname{tr}\left( {A} ({q}\mathbf{1}\mathbf{1}^T) \right)$ is a constant  and $u$ is a positive number, it is easy to deduce that two problems are equivalent.
\end{proof}
The scaling change of the affinity matrix is common. For example, if two pictures with different sizes contain the same objects, then the graphs created from the pictures may have the same structure. Assume that the scale ratio between two images is $u$. If the weight of the edge is the Euclidean distance between the corresponding nodes, then $A = u \tilde{A}$. 
Moreover, offset change may happen in the network graph matching problem. Let a graph $G$ represent a social network of a group at time $0$, and the weight of an edge is the duration of a relation between two individuals. The social network of a group at time $q$ is represented by $\tilde{G}$. If there is no structural change during this time period $q$, then $A + q\mathbf{1}\mathbf{1}^T = \tilde{A}$.

\section{Experiment}
\begin{table}[]
\centering
\label{tab:dataset}
\resizebox{\columnwidth}{!}{%
\begin{tabular}{lcccccc}
\hline
\textbf{Dataset} &
  $|V|$ &
  $|E|$ &
  \textbf{Attributed nodes} &
  \multicolumn{1}{l}{\textbf{Weighted edges}} &
  \textbf{Ground-truth} &
  \textbf{Description} \\ \hline
Real-world pictures &
  (100,1000) &
  (4 950, 499 500) &
  \textcolor{green}{\ding{51}} &
  \textcolor{green}{\ding{51}} &
  \textcolor{orange}{\ding{55}} &
  covering five common picture transformations \\
CMU House &
  (600,800) &
  (179 700, 319 600) &
  \textcolor{orange}{\ding{55}} &
  \textcolor{green}{\ding{51}} &
  \textcolor{orange}{\ding{55}} &
  a sequence of house pictures \\
Facebook-ego &
  4 039 &
  88 234 &
  \textcolor{orange}{\ding{55}} &
  \textcolor{orange}{\ding{55}} &
  \textcolor{green}{\ding{51}} &
  nodes: users; edges: relations \\
Protein network &
  1 004 &
  4 920 &
  \textcolor{orange}{\ding{55}} &
  \textcolor{orange}{\ding{55}} &
  \textcolor{green}{\ding{51}} &
  nodes: proteins; edges: interactions \\
Random generated graph &
  100 &
  300 &
  \textcolor{orange}{\ding{55}} &
  \textcolor{orange}{\ding{55}} &
  \textcolor{green}{\ding{51}} &
  delaunay triangulation \\ \hline
\end{tabular}%
}
\caption{Datasets used in experiments. $|V|$ is the number of nodes and  $|V|$ is the number of vertices. ( , ) represents a range.}
\end{table}
We evaluate the proposed algorithm CSGO\footnote{The code of CSGO is available at https://github.com/BinruiShen/CSGO.} and other contributions from the following aspects:

\begin{itemize}[leftmargin=1.5cm]
    \item Q1. Compared to another doubly stochastic constraining operator, how efficient is scalable softassign?
    \item Q2. Compared to other constrained gradient methods, what advancements does CSGO offer in attributed graph matching tasks?
    \item Q3. How robust is CSGO in plain graph matching tasks?
    \item Q4. To what extent do constrained gradient algorithms benefit from the adaptive step size parameter?
    \item Q5. How does the warm-start strategy benefit CSGO?
\end{itemize}

\textbf{ Setting} For CSGO, we set $\gamma = 10$ for attributed graph matching tasks and $\gamma = 60$ for plain graph matching tasks. With respect to the regularization parameter $\lambda$, according to \cite{lu2016fast}, the result is not sensitive to $\lambda$. For simplicity, we always use $\lambda =1$ in this paper.  All experiments are done in Python 3 with a single thread in an i7 2.80 GHz PC.

\textbf{Criteria} We evaluate the accuracy of algorithms in different cases with three criteria. 
\begin{itemize}[leftmargin=1.5cm]
\item Matching error:
\begin{equation}
    \frac{1}{2}\left\|A-M \widetilde{A} M^{T}\right\|_{Fro}+\left\|F-M \tilde{F}\right\|_{Fro}.
    \label{eq:matching error}
\end{equation}
For graphs with only affinity matrices, the measurement only contains the first term of \eqref{eq:matching error}.
\item Error rate: when matching errors of algorithms are close, it is better to use the error rate: the ratio between the matching error of an algorithm and the proposed CSGO.
\item Accuracy: $\frac{n_c}{n}$ where $n_c$ represents the number of correct matching nodes.
\end{itemize}


\textbf{Baselines} include constrained gradient algorithms such as DSPFP \cite{lu2016fast}, GA \cite{gold1996graduated}, and AIPFP \cite{leordeanu2009integer,lu2016fast}; optimal transport methods such as GWL \cite{xu2019gromov} and S-GW \cite{xu2019scalable}; and a spectral-based algorithm, GRASP \cite{hermanns2023grasp}. It is important to note that GA is a softassign-based algorithm.  Table \ref{tab:dataset} gathers the characteristics of the data sets used in this paper. Further details about these algorithms are shown as follows.
\begin{itemize}[leftmargin=1.5cm]
    \item DSPFP \cite{lu2016fast} is a fast doubly stochastic projected fixed-point method with an alternating projection.
    \item GA \cite{gold1996graduated} is a softassign-based method with an outer annealing process. 
    \item AIPFP \cite{leordeanu2009integer,lu2016fast} is an integer projected fixed point method with a fast greedy integer projection.
    \item GWL \cite{xu2019gromov} measures the distance between two graphs by Gromov-Wasserstein discrepancy and matches graphs by optimal transport.
    \item S-GWL \cite{xu2019scalable} is a scalable variant of GWL. It divides matching graphs into small graphs to match. (In Facebook network matching, we modify a parameter of S-GWL, increasing beta from 0.025 to 1, to avoid failure of partitioning; otherwise, S-GWL becomes a very slow GWL).
    \item GRASP \cite{hermanns2023grasp} aligns nodes based on functions derived from Laplacian matrix eigenvectors.
\end{itemize}
Unless otherwise specified, these algorithms are executed using the default parameters defined in the papers.



\subsection{Comparisons between doubly stochastic constraining operators}
This subsection compares the scalable softassign and the alternating projection in different magnitudes. The profit matrix $X$ is created by
\begin{equation}
    X = \phi X_{ran},
\end{equation}
where $X_{ran}$ a random matrix and $\phi$ represents magnitude. This set of experiments considers $\phi = [1,10,100]$. $\phi = 1$ is common for plain graph matching problems;  $\phi$ varies widely in attributed graph matching problems.

\begin{figure}[htbp]
    \centering
    \subfigure[$\phi =1$]{
        \includegraphics[width=0.3\textwidth]{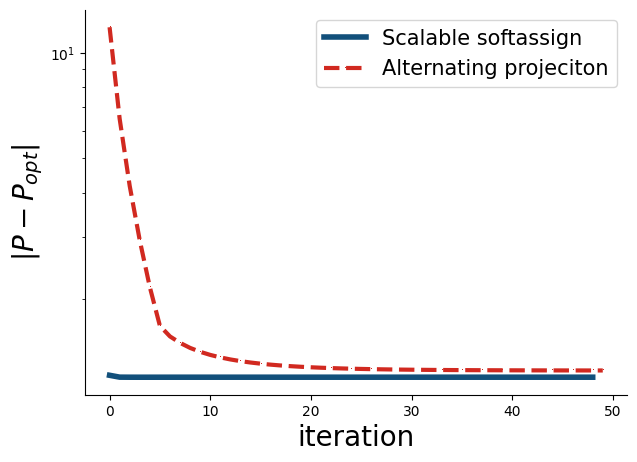}
    }
    \subfigure[$\phi =10$]{
	\includegraphics[width=0.3\textwidth]{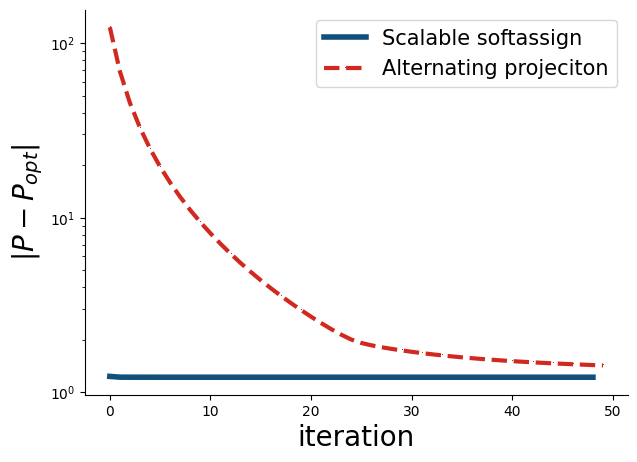}
    }
    \subfigure[$\phi =100$]{
	\includegraphics[width=0.3\textwidth]{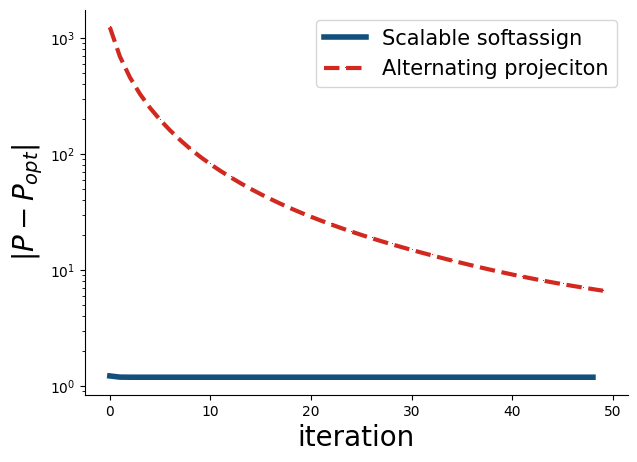}
    }
    \caption{The plots of differences between the optimal solution and the updated matrix with respect to a number of iterations. Here, $P$ is the matrix computed at the current iteration; $P_{opt}$ is the optimal solution computed by the simplex method.}
    \label{Fig.operators}
\end{figure}

Figure \ref{Fig.operators} demonstrates that the scalable softassign algorithm is invariant to magnitude. When $\phi=1$, the alternating projection converges rapidly. However, in the other two cases, the alternating projection does not converge within 50 iterations. 


\subsection{Algorithms with different constraining operators}
In this subsection, we compare the performance of the proposed CSGO with other state-of-the-art constrained gradient algorithms: GA \cite{gold1996graduated}, AIPFP \cite{leordeanu2009integer, lu2016fast} and  DSPFP \cite{lu2016fast}. Graphs are extracted from real-world images and house image sequences.  To avoid the influence of the step size parameter, we set $\alpha = 1$ for all algorithms, which is almost always the best choice in such two tasks.
\subsubsection{Real-world images}
In this set of experiments, the attributed graphs are constructed from a public dataset\footnote{http://www.robots.ox.ac.uk/~vgg/research/affine/}. This data set, which contains eight sets of pictures, covers five common picture transformations: viewpoint changes, scale changes, image blur, JPEG compression, and illumination. 

The construction consists of three primary steps.

\begin{enumerate}[label=(\arabic*), leftmargin=1.5cm] 
    \item  Node extraction: the scale-invariant feature transform method (SIFT) \cite{lowe2004distinctive} is used to identify key points as potential nodes and extract the corresponding feature vectors.
    \item  Node selection: the selected nodes should exhibit a high degree of similarity (i.e. the inner product of feature vectors) to all the other candidate nodes of the other graph.
    \item Edge weight calculation: these nodes are fully connected by edges with weights, each with a weight representing the Euclidean distance between the corresponding nodes.
\end{enumerate}

 The running time and matching error are calculated by the average results of five matching pairs (1 vs. 2, 2 vs. 3 \dots 5 vs. 6) from the same picture set (except set \textit{ bike}, we only record the results of 1 vs. 2, as other pictures in the set do not have enough key points). The matching results are visualized in Figure \ref{Fig.real matching}. The numerical results are shown in Figure \ref{Fig.time} and Figure \ref{Fig.error}. 
 
\begin{figure}[htbp]
    \centering
    \subfigure[Boat]{
        \includegraphics[width=0.8\textwidth]{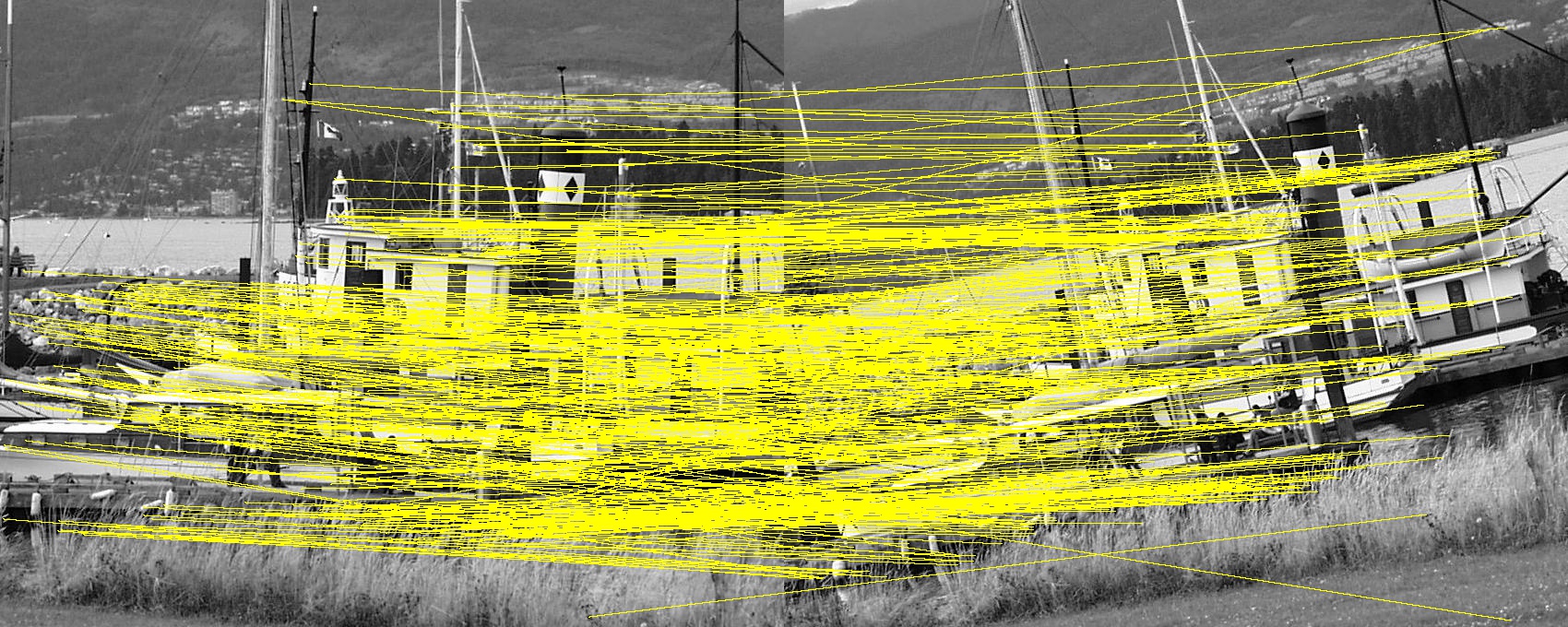}
    }

    \subfigure[Wall]{
	\includegraphics[width=0.8\textwidth]{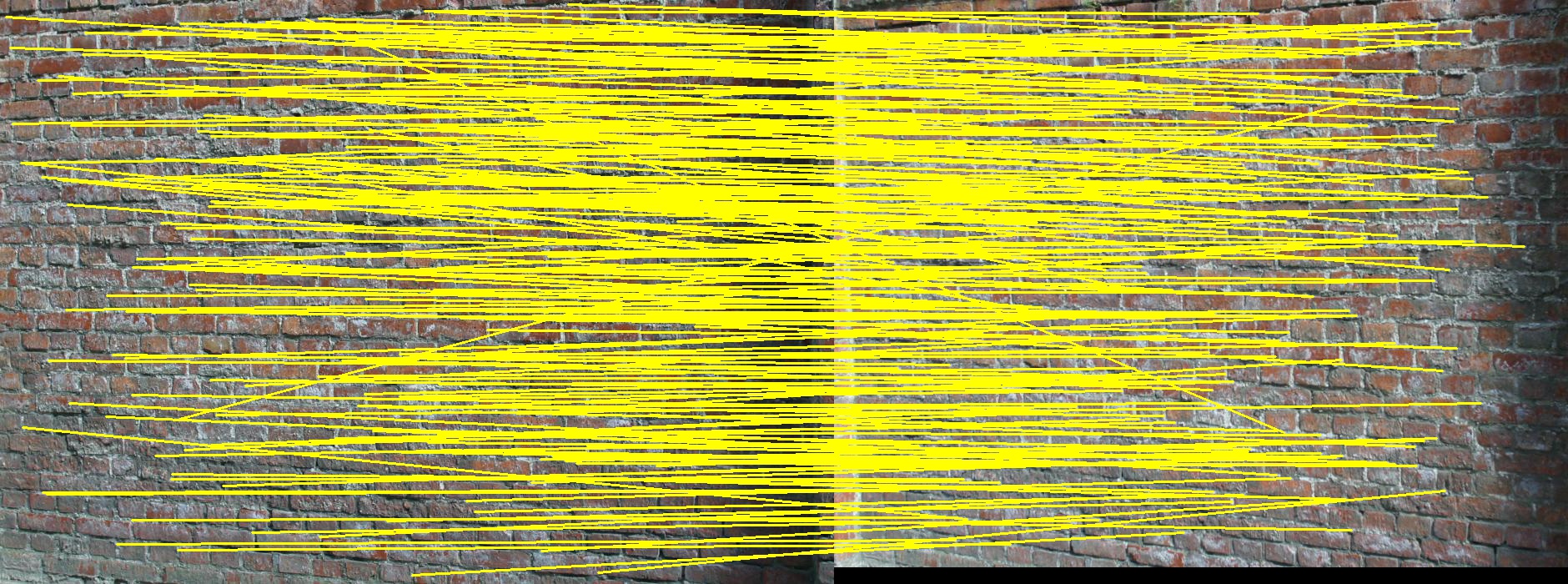}
    }
    \caption{Graphs from real-world images.}
    \label{Fig.real matching}
\end{figure}

\begin{figure}[h]
\centering  
\includegraphics[width=1\textwidth]{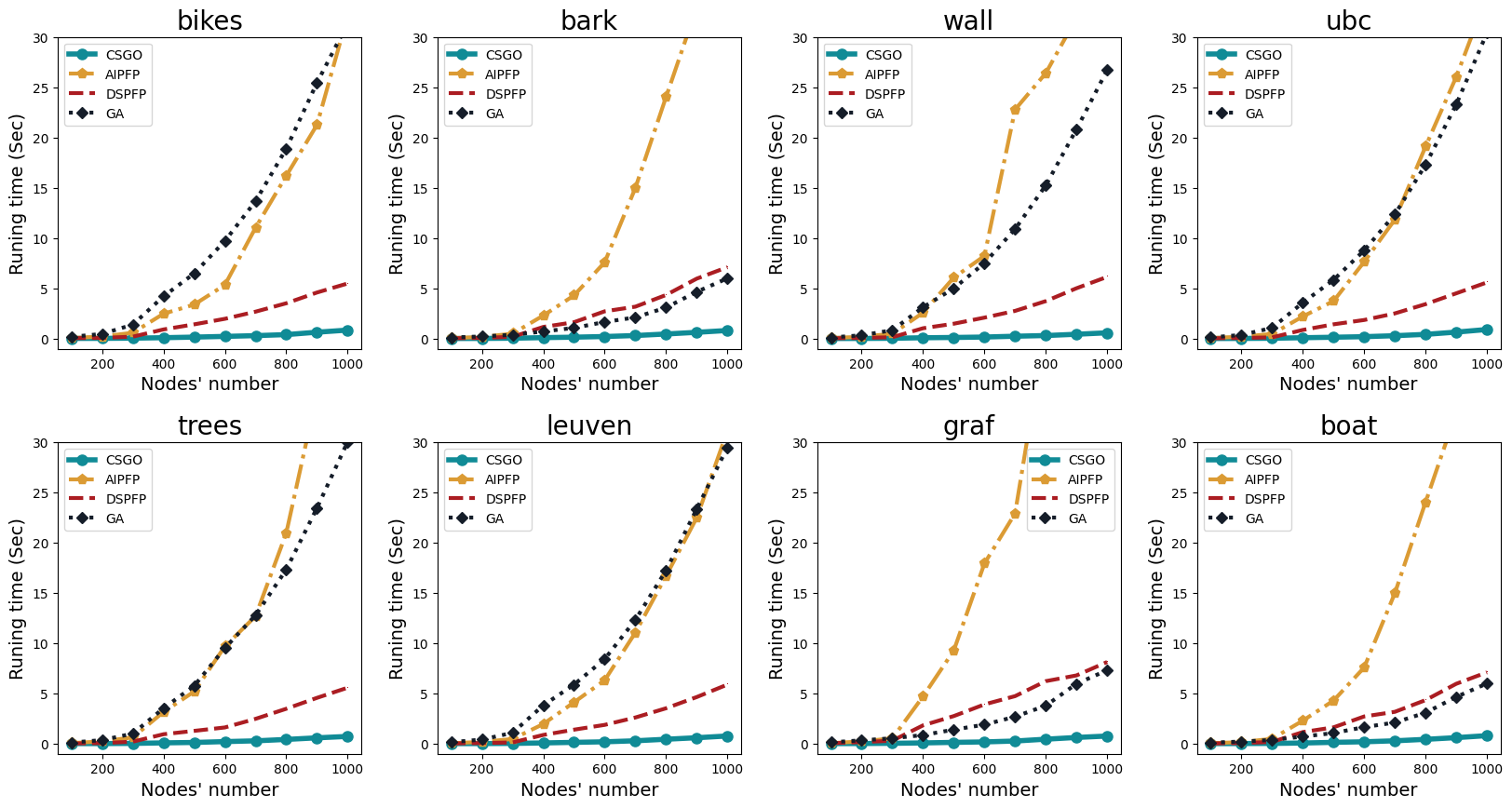} 
\caption{Matching time of constrained gradient algorithms in attributed graph matching experiments.}
\label{Fig.time}
\end{figure}

\begin{figure}[h]
\centering  
\includegraphics[width=1\textwidth]{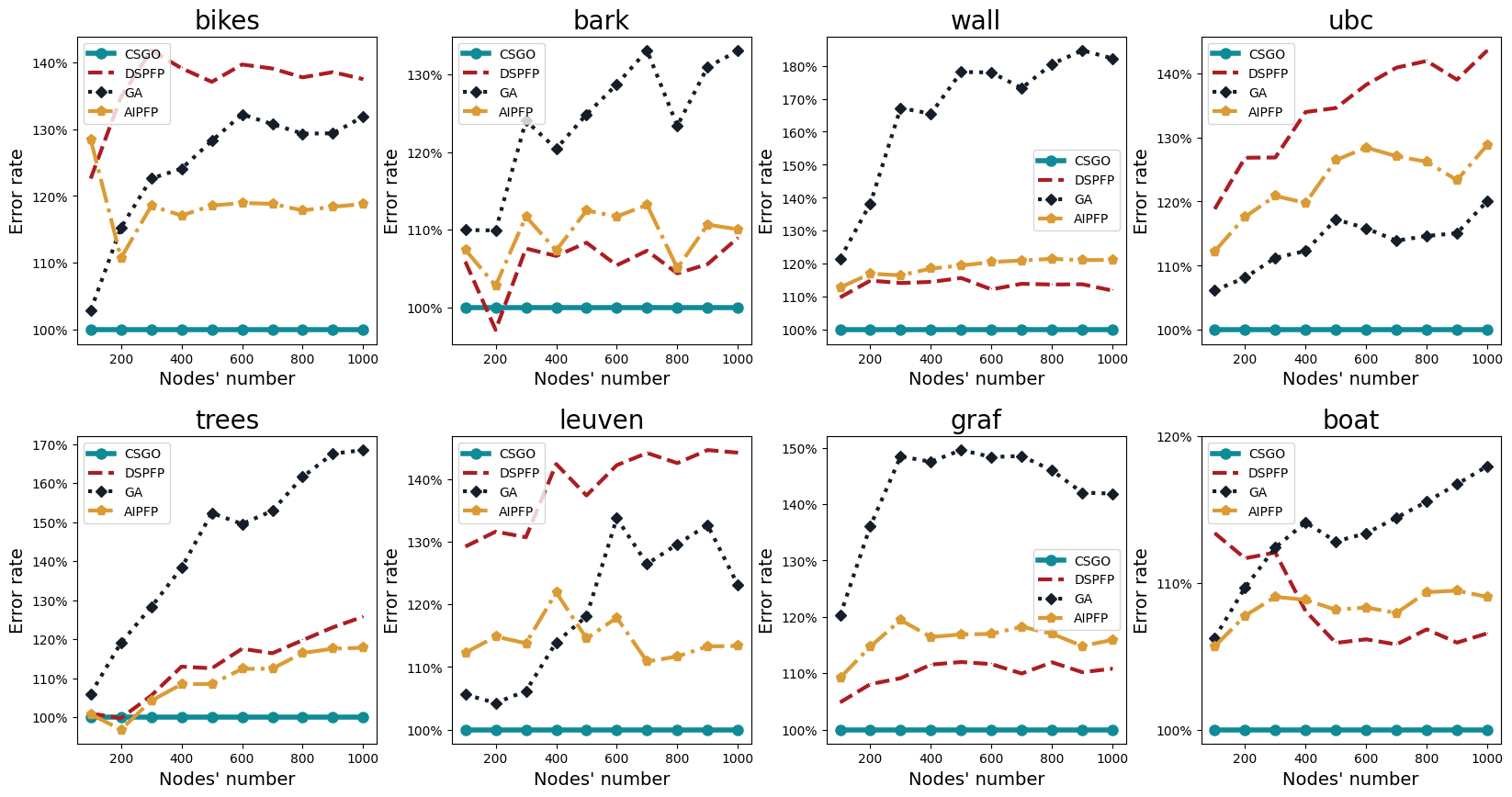} 
\caption{Matching error of constrained gradient algorithms in attributed graph matching experiments.}
\label{Fig.error}
\end{figure}

Table \ref{tab:real-world} clearly shows that CSGO is at least ten times faster than other algorithms for this task. Figure \ref{Fig.time} further illustrates that CSGO consistently exhibits a significant efficiency advantage across all experimental groups and scales. Additionally, Table \ref{tab:real-world} indicates that, on average, CSGO's matching error is at least 10\% lower than that of other algorithms. Figure \ref{Fig.error} provides a detailed comparison, showing the overall superiority of CSGO over other algorithms in each experimental group.



\subsubsection{Image sequence}

Graphs with weighted edges are constructed from the CMU House sequence images\footnote{http://vasc.ri.cmu.edu/idb/html/motion/} in this set of experiments. The SIFT extracts the key points  (around 700) as nodes of the graph. The Euclidean distance between these nodes serves as the weight of the corresponding edge.  Matching pairs consist of the first image and subsequent images with 5 image sequence gaps (such as image 1 vs. image 6 and so on). See Figure \ref{Fig.House matching} to visualize the matching results. 

Compared to the previous matching task, Table \ref{tab:house} shows that CSGO's advantage over other algorithms has further increased, with at least a 15\% improvement in accuracy and a tenfold increase in efficiency. Figure \ref{Fig.House experiments} illustrates that CSGO consistently achieves the lowest error in most cases while maintaining a significant lead in efficiency.

\begin{figure}[htbp]
    \centering
    \subfigure[10\% links]{
        \includegraphics[width=0.4\textwidth]{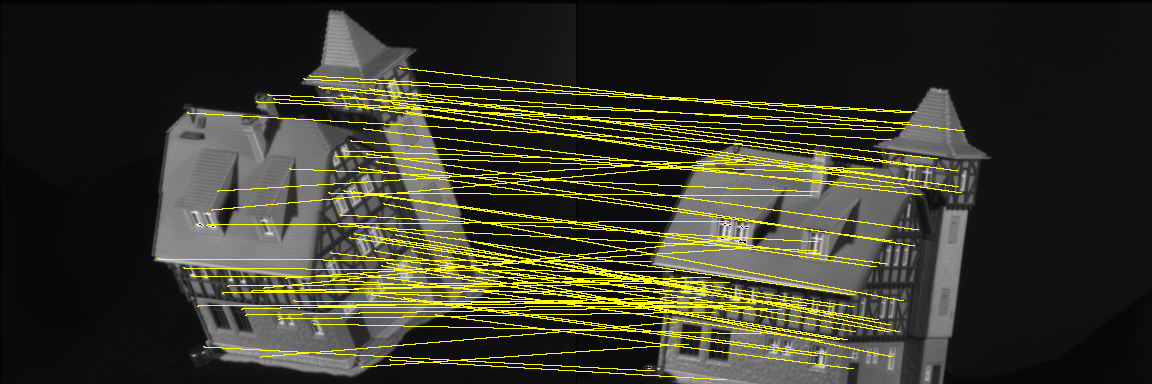}
    }
    \subfigure[30\% links]{
	\includegraphics[width=0.4\textwidth]{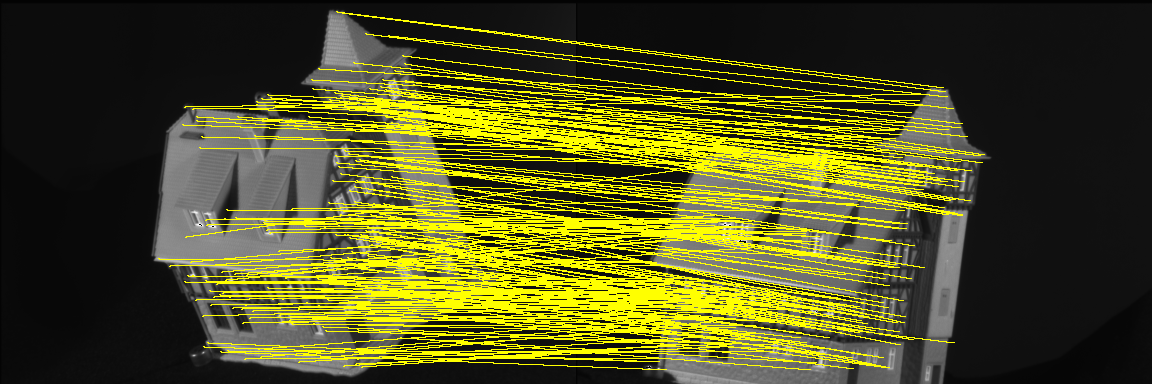}
    }
    \quad    
    \subfigure[50\% links]{
    	\includegraphics[width=0.4\textwidth]{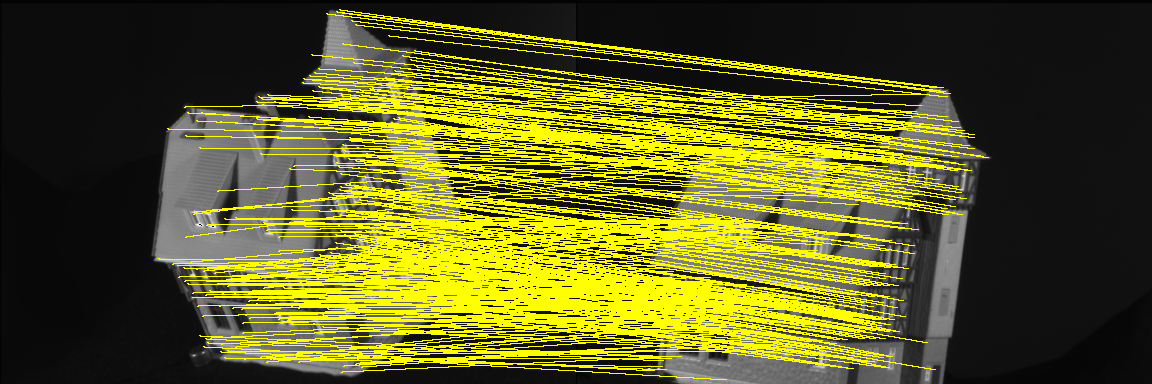}
    }
    \subfigure[100\% links]{
	\includegraphics[width=0.4\textwidth]{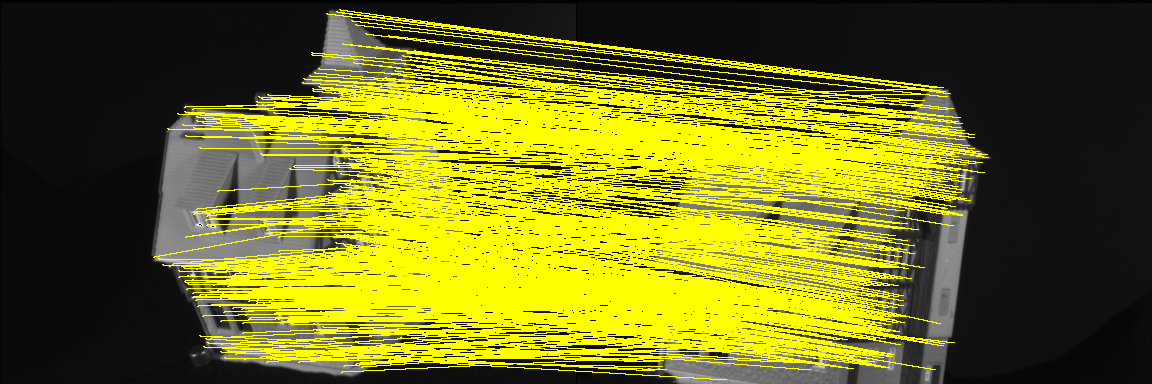}
    }
    \caption{House matching results}
    \label{Fig.House matching}
\end{figure}

\begin{table}[ht]
\centering
\caption{comparisons between algorithms on graphs from house sequence.}
\label{tab:house}
\begin{tabular}{lcccccccc}
\hline
               & DSPFP/CSGO &  &  & GA/CSGO &  &  & AIPFP/CSGO &  \\ \hline
Running time     & 10.0X       &  &  & 30.9X       &  &  & 43.7X     &  \\
Matching error & 1.11       &  &  & 1.12       &  &  & 1.27     &  \\ \hline
\end{tabular}
\end{table}

\begin{table}[ht]
\centering
\caption{comparisons between algorithms on graphs from real-world pictures.}
\label{tab:real-world}
\begin{tabular}{lcccccccc}
\hline
               & DSPFP/CSGO &  &  & GA/CSGO &  &  & AIPFP/CSGO &  \\ \hline
Running time     & 10.3X       &  &  & 31.1X     &  &  & 57.9X       &  \\
Matching error & 1.20       &  &  & 1.35     &  &  & 1.15     &  \\ \hline
\end{tabular}
\end{table}

\begin{figure}[H]
\centering  
\includegraphics[width=0.8\textwidth]{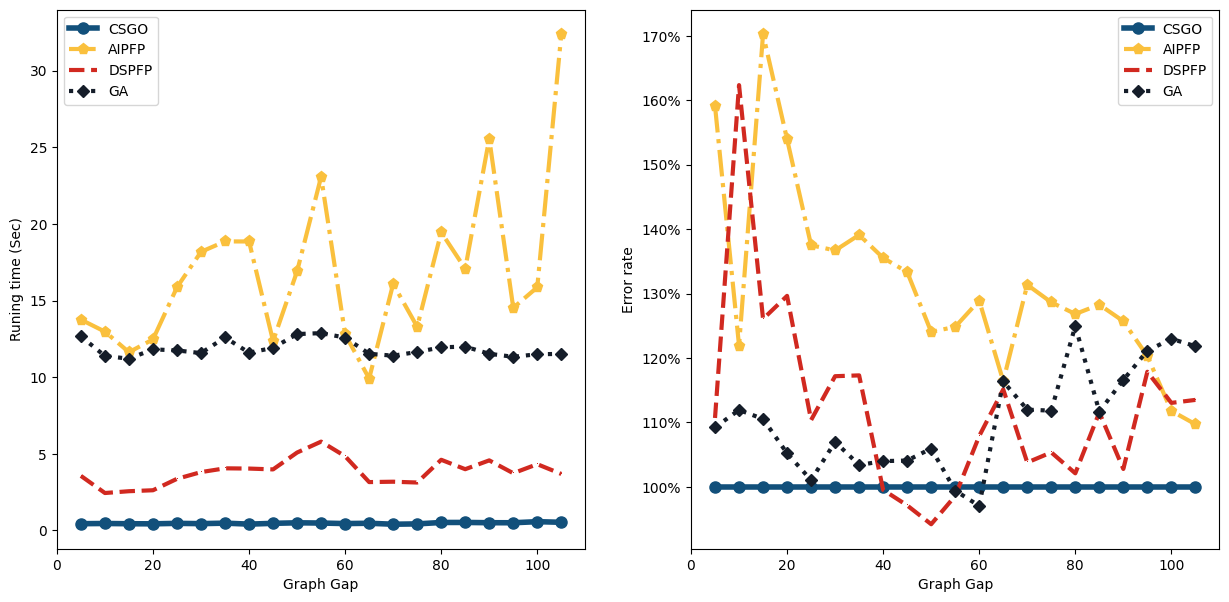} 
\caption{Graphs from house sequence.}
\label{Fig.House experiments}
\end{figure}



\subsection{Plain graphs}
This set of experiments evaluates CSGO and other matching algorithms by two sets of plain graphs:
the yeast's protein-protein interaction (PPI) networks contain 1,004 proteins and 4,920 high-confidence interactions\footnotemark[4]; the social network comprising 'circles' (or 'friends lists') from Facebook \cite{snapnets} contains 4039 users (nodes) and 88234 relations (edges). Following the experimental protocol of \cite{xu2019scalable}, we compare different methods on matching networks with $5\%$, $15\%$ and $25\%$ noisy versions. Table \ref{tab:PPI} and Table \ref{tab:facebook} list the performance of various methods. In both tasks, CSGO outperforms other algorithms, showing approximately a 20\% improvement in accuracy, and ranks among the fastest in terms of speed.
\begin{table}[ht]
\centering
\caption{Comparisons on yeast PPI}
\label{tab:PPI}
\resizebox{\columnwidth}{!}{%
\begin{tabular}{c|ccccccc}
\hline
Yeast network & \multicolumn{2}{c}{5\% noise} & \multicolumn{2}{c}{15\% noise} & \multicolumn{2}{c}{25\% noise} \\ \hline
Methods & Node Acc & time  & Node Acc & time  & Node Acc & time  \\ \hline
GRASP \cite{hermanns2023grasp}     & 38.6\%   & \textbf{1.1}s    & 8.3\%   & \textbf{1.2}s    & 5.6\%   & \textbf{1.2}s    \\
S-GWL \cite{xu2019scalable}   & 81.3\%   & 82.3s    & 62.4\%   & 82.1s    & 55.5\%  & 88.4s   \\
GWL\cite{xu2019gromov}   & 83.7\%   & 226.4s    & 66.3\%   & 254.7s    & 57.6\%  & 246.5s   \\
GA \cite{gold1996graduated}      & 14.0\%   & 24.4s    & 9.6\%    & 24.5s    & 7.4\%    & 24.0s    \\
DSPFP \cite{lu2016fast}   & 78.1\%   & 10.2s    & 60.8\%   & 
10.14s     & 42.9\%   & 9.8s       \\
AIPFP \cite{leordeanu2009integer,lu2016fast}     & 43.1\%   & 105.4s    & 27.1\%   & 75.2s   & 22.1\%   & 73.8s\\ \hline
CSGO      & \textbf{91.3\%}   & 4.3s    & \textbf{85.0\%}    & 4.4s    & \textbf{80.7\%}    & 4.6s       \\ \hline
\end{tabular}%
}
\end{table}
\begin{table}[ht]
\centering
\caption{Comparisons on Facebook network}
\label{tab:facebook}
\resizebox{\columnwidth}{!}{%
\begin{tabular}{c|ccccccc}
\hline
Social network & \multicolumn{2}{c}{5\% noise} & \multicolumn{2}{c}{15\% noise} & \multicolumn{2}{c}{25\% noise} \\ \hline
Methods & Node Acc & time  & Node Acc & time  & Node Acc & time  \\ \hline
GRASP\cite{hermanns2023grasp}     & 37.9\%   & \textbf{63.6s}    & 20.3\%   & \textbf{67.4s}    & 15.7\%   & \textbf{71.3s}    
\\ 
S-GWL \cite{xu2019scalable}   &  26.4\%   &  1204.1s    &  18.3\%   &  1268.2s    &  17.9\%  &  1295.8s   \\
GWL\cite{xu2019gromov}   &  78.1\%   &  3721.6s    &  68.4\%   &  4271.3s    &  60.8\%  &  4453.9s   \\
GA \cite{gold1996graduated}      & 35.5\%   & 793.2s    & 21.4\%    & 761.7s    & 16.0\%    & 832.6s    \\
DSPFP \cite{lu2016fast}   &  79.7\%   & 151.3s    & 68.3\%   & 
154.2s     & 62.2\%   & 156.9s       \\
AIPFP \cite{leordeanu2009integer,lu2016fast}     & 68.6\%   & 2705.5s    & 55.1\%   &2552.7s   & 47.8\%   & 2513.8s    \\ \hline CSGO       & \textbf{91.1\%}   & 87.3s     & \textbf{88.3\%}    &  88.3s  & \textbf{86.3\%}    & 89.1s         \\ \hline
\end{tabular}%
}
\end{table}

\subsection{The effect of adaptive step size parameter}
This set of experiments investigates the impact of employing the adaptive step size parameter across four algorithms that employ different constraining operators. Random graphs are generated as follows: (1) generating $n$ pairs of numbers representing the coordinates of nodes; (2) connecting these nodes by the Delaunay triangulation method; (3) assigning Euclidean distances between nodes as weights for the corresponding edges. Its noisy version contains $q\%$ nodes' deletion, where $ q \in \{1, 2, 3, 4, 5\}$. Each experiment is repeated one hundred times to account for randomness.
The matching results are shown in Figure \ref{Fig.adaptive} and Table \ref{tab:improvements}.
\begin{figure}[h]
\centering  
\includegraphics[width=1\textwidth,height= 2in]{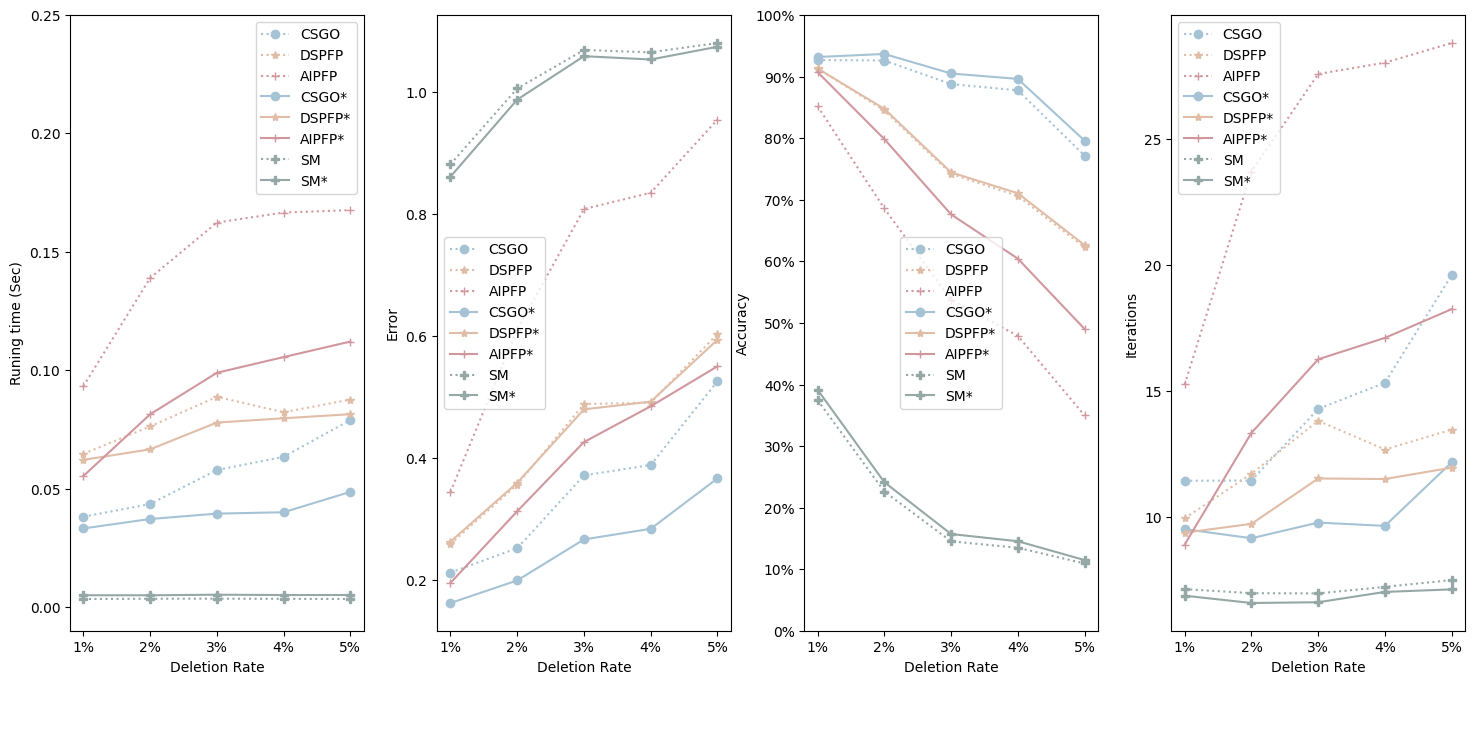} 
\caption{The effect of adaptive step size parameter in noisy graph matching. Algorithms with * represent algorithms equipped with the adaptive step size parameter.}
\label{Fig.adaptive}
\end{figure}

\begin{table}[]
\centering
\caption{Improvements from the adaptive step size parameter}
\label{tab:improvements}
\begin{tabular}{cccccc}
\hline
Algorithms & Constraining operator   & Iteration   & Time   & Matching error & Accuracy \\ \hline
CSGO       & scalable softassign & 30.3\% & 29.5\%  & 27.8\%           & 1.4\%     \\
DSPFP      & alternating projeciton & 13.9\% & 11\%   & 1\%           & 0.4\%     \\
AIPFP & greedy method  & 40.0\% & 49.5\%  & 43.7\%          & 21.7\%    \\
SM         & spectral normalization & 4.6\% & -44.6\% & 0.8\%           & 2.7\%    \\ \hline
\end{tabular}
\end{table}
Figure \ref{Fig.adaptive} and Table \ref{tab:improvements} show that the adaptive step size parameter strategy enhances the accuracy of constrained gradient algorithms. Besides, this strategy reduces all algorithms' iterations. Most of the algorithms' efficiency benefits from this, except SM whose constraining operator is cheap to compute. Among algorithms with doubly stochastic constraints, the proposed CSGO consistently performs better.

\subsection{The effect of warm-start strategy}

This set of experiments assesses the benefits of the warm-start method for the CSGO algorithm in the task of attributed graph matching. We evaluated this strategy using the attributed graphs from real-world images, each with four distinct node quantities.  The results, shown in Figure \ref{Fig.warm}, suggest that as the number of nodes increases, the acceleration effect becomes more pronounced. This is likely due to the computational cost of the gradient matrix increasing more rapidly compared to other modules. We estimate that the maximum acceleration achievable by this method is approximately 30\%, as CSGO typically requires only 3 to 5 iterations to converge in attributed graph matching.

\begin{figure}[h]
\centering  
\includegraphics[width=0.4\textwidth]{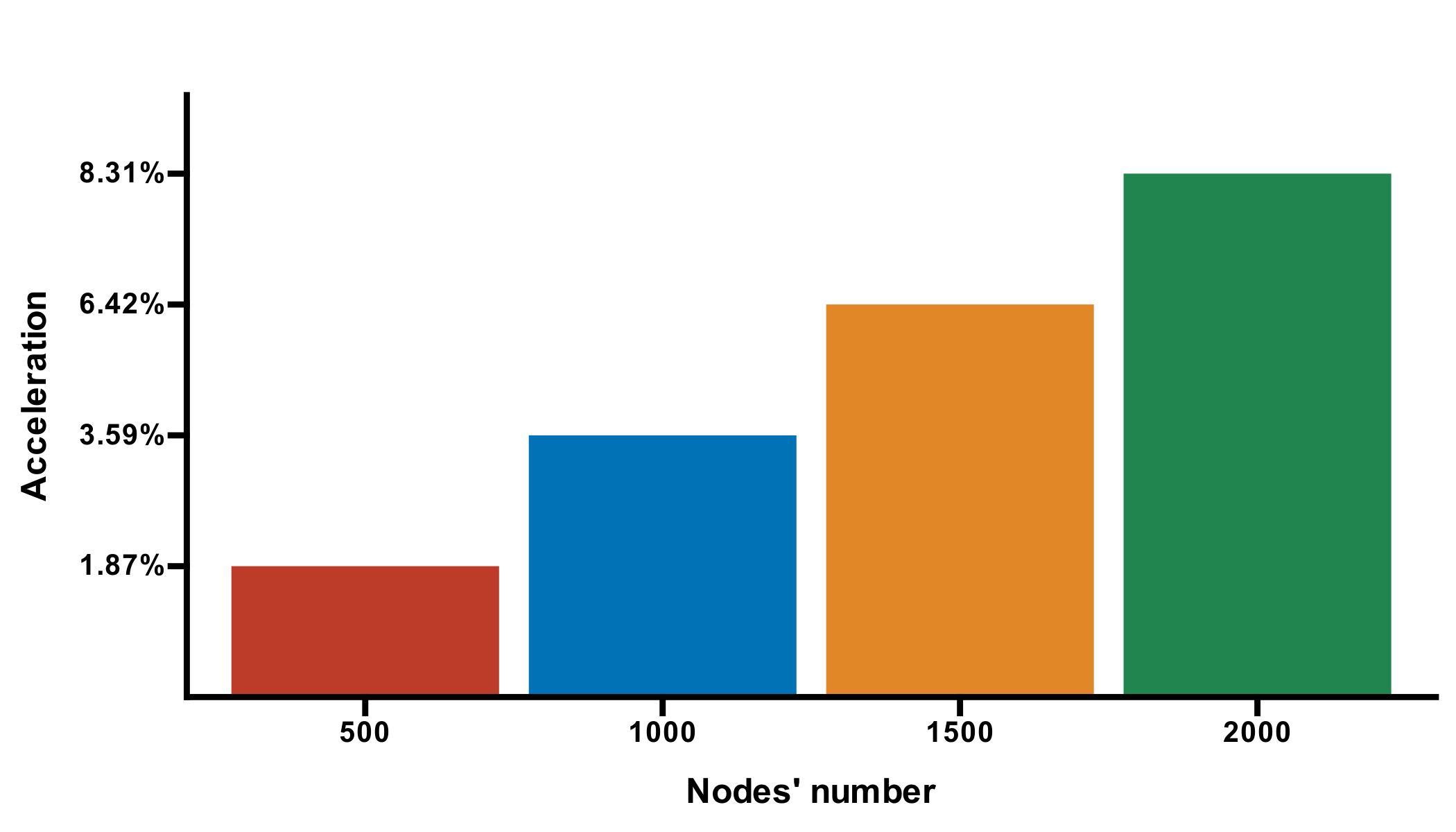} 
\caption{The benefits of the warm-start method for the CSGO algorithm in the task of attributed graph matching.}
\label{Fig.warm}
\end{figure}

\section{Conclusions}
We propose a constrained-softassign gradient optimization for large graph matching problems by a scalable softassign, the adaptive step size parameter, and a warm-start strategy. Scalable softassign is robust to the nodes' number and avoids the risk of overflow based on rigorous mathematical analysis. The adaptive step size parameter can guarantee the algorithms' convergence and enhance their robustness and efficiency. Compared with other state-of-the-art algorithms, CSGO achieves high levels of both speed and accuracy.

A limitation of this study is that the constrained gradient algorithms may converge to a local maximum, and the optimality of the solutions cannot be guaranteed. Another limitation is the algorithms' sensitivity to variable initialization, as they lack a robust initialization strategy. Therefore, our future work will focus on exploring methods to avoid or escape local maxima and improving the initialization.

%

{\small

\bibliography{0.Main.bib}

\begin{thebibliography}{41}
\providecommand{\natexlab}[1]{#1}
\providecommand{\url}[1]{\texttt{#1}}
\expandafter\ifx\csname urlstyle\endcsname\relax
  \providecommand{\doi}[1]{doi: #1}\else
  \providecommand{\doi}{doi: \begingroup \urlstyle{rm}\Url}\fi

\bibitem[Altschuler et~al.(2017)Altschuler, Niles-Weed, and
  Rigollet]{altschuler2017near}
Jason Altschuler, Jonathan Niles-Weed, and Philippe Rigollet.
\newblock {Near-linear time approximation algorithms for optimal transport via
  Sinkhorn iteration}.
\newblock In \emph{Advances in Neural Information Processing Systems},
  volume~30, 2017.

\bibitem[Blanchard et~al.(2021)Blanchard, Higham, and
  Higham]{blanchard2021accurately}
Pierre Blanchard, Desmond~J Higham, and Nicholas~J Higham.
\newblock Accurately computing the log-sum-exp and softmax functions.
\newblock \emph{IMA Journal of Numerical Analysis}, 41\penalty0 (4):\penalty0
  2311--2330, 2021.

\bibitem[Caelli and Kosinov(2004)]{caelli2004inexact}
Terry Caelli and Serhiy Kosinov.
\newblock Inexact graph matching using eigen-subspace projection clustering.
\newblock \emph{International Journal of Pattern Recognition and Artificial
  Intelligence}, 18\penalty0 (03):\penalty0 329--354, 2004.

\bibitem[Chang et~al.(2021)Chang, Shao, Zhang, and Zhang]{chang2021lovasz}
Kung-Ching Chang, Sihong Shao, Dong Zhang, and Weixi Zhang.
\newblock Lov{\'a}sz extension and graph cut.
\newblock \emph{Communications in Mathematical Sciences}, 19\penalty0
  (3):\penalty0 761--786, 2021.

\bibitem[Chen and Grauman(2012)]{chen2012efficient}
Chao-Yeh Chen and Kristen Grauman.
\newblock Efficient activity detection with max-subgraph search.
\newblock In \emph{2012 IEEE Conference on Computer Vision and Pattern
  Recognition}, pages 1274--1281. IEEE, 2012.

\bibitem[Cho et~al.(2010)Cho, Lee, and Lee]{2010Reweighted}
Minsu Cho, Jungmin Lee, and Kyoung~Mu Lee.
\newblock {Reweighted Random Walks for Graph Matching}.
\newblock In \emph{European Conference on Computer Vision}, 2010.

\bibitem[Cuturi(2013)]{cuturi2013Sinkhorn}
Marco Cuturi.
\newblock Sinkhorn distances: Lightspeed computation of optimal transport.
\newblock In \emph{Advances in Neural Information Processing Systems},
  volume~26, pages 2292--2300, 2013.

\bibitem[Egozi et~al.(2012)Egozi, Keller, and Guterman]{egozi2012probabilistic}
Amir Egozi, Yosi Keller, and Hugo Guterman.
\newblock A probabilistic approach to spectral graph matching.
\newblock \emph{IEEE Transactions on Pattern Analysis and Machine
  Intelligence}, 35\penalty0 (1):\penalty0 18--27, 2012.

\bibitem[Garey and Johnson(1979)]{garey1979computers}
Michael~R Garey and David~S Johnson.
\newblock \emph{{Computers and Intractability: A Guide to the} {Theory of
  NP-Completeness}}.
\newblock WH Freeman \& Co., 1979.

\bibitem[Gold and Rangarajan(1996)]{gold1996graduated}
Steven Gold and Anand Rangarajan.
\newblock A graduated assignment algorithm for graph matching.
\newblock \emph{IEEE Transactions on Pattern Analysis and Machine
  Intelligence}, 18\penalty0 (4):\penalty0 377--388, 1996.

\bibitem[Gordon et~al.(2020)Gordon, Hiatt, Bouhaddou, Rezelj, Ulferts, Braberg,
  Jureka, Obernier, Guo, Batra, et~al.]{gordon2020comparative}
David~E Gordon, Joseph Hiatt, Mehdi Bouhaddou, Veronica~V Rezelj, Svenja
  Ulferts, Hannes Braberg, Alexander~S Jureka, Kirsten Obernier, Jeffrey~Z Guo,
  Jyoti Batra, et~al.
\newblock Comparative host-coronavirus protein interaction networks reveal
  pan-viral disease mechanisms.
\newblock \emph{Science}, 370\penalty0 (6521):\penalty0 eabe9403, 2020.

\bibitem[Harville and David(2008)]{Harville2008Matrix}
Harville and A.~David.
\newblock \emph{{Matrix Algebra From a Statistician's Perspective}}.
\newblock Springer Science \& Business Media, 2008.

\bibitem[Hermanns et~al.(2023)Hermanns, Skitsas, Tsitsulin, Munkhoeva, Kyster,
  Nielsen, Bronstein, Mottin, and Karras]{hermanns2023grasp}
Judith Hermanns, Konstantinos Skitsas, Anton Tsitsulin, Marina Munkhoeva,
  Alexander Kyster, Simon Nielsen, Alexander~M Bronstein, Davide Mottin, and
  Panagiotis Karras.
\newblock {GRASP: Scalable Graph Alignment by Spectral Corresponding
  Functions}.
\newblock \emph{ACM Transactions on Knowledge Discovery from Data}, 17\penalty0
  (4):\penalty0 1--26, 2023.

\bibitem[Kosowsky and Yuille(1994)]{kosowsky1994invisible}
Jeffrey~J Kosowsky and Alan~L Yuille.
\newblock The invisible hand algorithm: Solving the assignment problem with
  statistical physics.
\newblock \emph{Neural Networks}, 7\penalty0 (3):\penalty0 477--490, 1994.

\bibitem[Kuhn(1955)]{kuhn1955hungarian}
Harold~W Kuhn.
\newblock {The Hungarian method for the assignment problem}.
\newblock \emph{Naval Research Logistics Quarterly}, 2\penalty0 (1-2):\penalty0
  83--97, 1955.

\bibitem[Lan et~al.(2022)Lan, Ma, Yu, Yuan, and Ma]{lan2022aednet}
Zixun Lan, Ye~Ma, Limin Yu, Linglong Yuan, and Fei Ma.
\newblock {AEDNet: Adaptive Edge-Deleting Network For Subgraph Matching}.
\newblock \emph{Pattern Recognition}, page 109033, 2022.

\bibitem[Lan et~al.(2024)Lan, Hong, Ma, and Ma]{lan2022more}
Zixun Lan, Binjie Hong, Ye~Ma, and Fei Ma.
\newblock {More Interpretable Graph Similarity Computation Via Maximum Common
  Subgraph Inference}.
\newblock \emph{IEEE Transactions on Knowledge and Data Engineering}, pages
  1--12, 2024.

\bibitem[Lawler(1963)]{1963The}
E.~L. Lawler.
\newblock {The Quadratic Assignment Problem}.
\newblock \emph{Management Science}, 9\penalty0 (4):\penalty0 586--599, 1963.

\bibitem[Leordeanu and Hebert(2005)]{leordeanu2005spectral}
Marius Leordeanu and Martial Hebert.
\newblock A spectral technique for correspondence problems using pairwise
  constraints.
\newblock In \emph{Tenth IEEE International Conference on Computer Vision
  (ICCV'05) Volume 1}, volume~2, pages 1482--1489. IEEE, 2005.

\bibitem[Leordeanu et~al.(2009)Leordeanu, Hebert, and
  Sukthankar]{leordeanu2009integer}
Marius Leordeanu, Martial Hebert, and Rahul Sukthankar.
\newblock An integer projected fixed point method for graph matching and map
  inference.
\newblock In \emph{Advances in Neural Information Processing Systems}, pages
  1114--1122, 2009.

\bibitem[Leskovec and Krevl(2014)]{snapnets}
Jure Leskovec and Andrej Krevl.
\newblock {SNAP Datasets}: {Stanford} large network dataset collection.
\newblock \url{http://snap.stanford.edu/data}, June 2014.

\bibitem[Lowe(2004)]{lowe2004distinctive}
David~G Lowe.
\newblock Distinctive image features from scale-invariant keypoints.
\newblock \emph{International Journal of Computer Vision}, 60\penalty0
  (2):\penalty0 91--110, 2004.

\bibitem[Lu et~al.(2016)Lu, Huang, and Liu]{lu2016fast}
Yao Lu, Kaizhu Huang, and Cheng-Lin Liu.
\newblock A fast projected fixed-point algorithm for large graph matching.
\newblock \emph{Pattern Recognition}, 60:\penalty0 971--982, 2016.

\bibitem[Maron and Lipman(2018)]{maron2018probably}
Haggai Maron and Yaron Lipman.
\newblock (probably) concave graph matching.
\newblock In \emph{Advances in Neural Information Processing Systems},
  volume~31, 2018.

\bibitem[MH~Nguyen et~al.(2024)MH~Nguyen, Nguyen, Diep, Pham, Cao, Nguyen,
  Swoboda, Ho, Albarqouni, Xie, et~al.]{mh2024lvm}
Duy MH~Nguyen, Hoang Nguyen, Nghiem Diep, Tan~Ngoc Pham, Tri Cao, Binh Nguyen,
  Paul Swoboda, Nhat Ho, Shadi Albarqouni, Pengtao Xie, et~al.
\newblock {Lvm-med: Learning large-scale self-supervised vision models for
  medical imaging via second-order graph matching}.
\newblock \emph{Advances in Neural Information Processing Systems}, 36, 2024.

\bibitem[Michel et~al.(2011)Michel, Oikonomidis, and Argyros]{2011Scale}
D.~Michel, I.~Oikonomidis, and A.~Argyros.
\newblock Scale invariant and deformation tolerant partial shape matching.
\newblock \emph{Image and Vision Computing}, 29\penalty0 (7):\penalty0
  459--469, 2011.

\bibitem[Nguyen et~al.(2024)Nguyen, Diep, Nguyen, Le, Nguyen, Nguyen, Nguyen,
  Ho, Xie, Wattenhofer, Zhou, Sonntag, and
  Niepert]{nguyen2024logramedlongcontextmultigraph}
Duy M.~H. Nguyen, Nghiem~T. Diep, Trung~Q. Nguyen, Hoang-Bao Le, Tai Nguyen,
  Tien Nguyen, TrungTin Nguyen, Nhat Ho, Pengtao Xie, Roger Wattenhofer, James
  Zhou, Daniel Sonntag, and Mathias Niepert.
\newblock {LoGra-Med: Long Context Multi-Graph Alignment for Medical
  Vision-Language Model}, 2024.
\newblock URL \url{https://arxiv.org/abs/2410.02615}.

\bibitem[Peyr{\'e} et~al.(2017)Peyr{\'e}, Cuturi,
  et~al.]{peyre2017computational}
Gabriel Peyr{\'e}, Marco Cuturi, et~al.
\newblock {Computational Optimal Transport}.
\newblock Working Papers 2017-86, Center for Research in Economics and
  Statistics, October 2017.
\newblock URL \url{https://ideas.repec.org/p/crs/wpaper/2017-86.html}.

\bibitem[Robles-Kelly and Hancock(2007)]{robles2007riemannian}
Antonio Robles-Kelly and Edwin~R Hancock.
\newblock {A Riemannian approach to graph embedding}.
\newblock \emph{Pattern Recognition}, 40\penalty0 (3):\penalty0 1042--1056,
  2007.

\bibitem[Rontsis and Goulart(2020)]{rontsis2020optimal}
Nikitas Rontsis and Paul Goulart.
\newblock Optimal approximation of doubly stochastic matrices.
\newblock In \emph{International Conference on Artificial Intelligence and
  Statistics}, pages 3589--3598. PMLR, 2020.

\bibitem[Shen et~al.(2020)Shen, Niu, and Zhu]{shen2020fabricated}
Binrui Shen, Qiang Niu, and Shengxin Zhu.
\newblock {Fabricated Pictures Detection with Graph Matching}.
\newblock In \emph{Proceedings of the 2020 2nd Asia Pacific Information
  Technology Conference}, pages 46--51, 2020.

\bibitem[Sinkhorn and Knopp(1967)]{Sinkhorn1967concerning}
Richard Sinkhorn and Paul Knopp.
\newblock Concerning nonnegative matrices and doubly stochastic matrices.
\newblock \emph{Pacific Journal of Mathematics}, 21\penalty0 (2):\penalty0
  343--348, 1967.

\bibitem[Song et~al.(2023)Song, Wei, Bai, Yang, and Jia]{song2023graphalign}
Ziying Song, Haiyue Wei, Lin Bai, Lei Yang, and Caiyan Jia.
\newblock {Graphalign: Enhancing accurate feature alignment by graph matching
  for multi-modal 3d object detection}.
\newblock In \emph{Proceedings of the IEEE/CVF International Conference on
  Computer Vision}, pages 3358--3369, 2023.

\bibitem[Umeyama(1988)]{umeyama1988eigendecomposition}
Shinji Umeyama.
\newblock An eigendecomposition approach to weighted graph matching problems.
\newblock \emph{IEEE Transactions on Pattern Analysis and Machine
  Intelligence}, 10\penalty0 (5):\penalty0 695--703, 1988.

\bibitem[Wang et~al.(2017)Wang, Ling, Lang, and Feng]{wang2017graph}
Tao Wang, Haibin Ling, Congyan Lang, and Songhe Feng.
\newblock Graph matching with adaptive and branching path following.
\newblock \emph{IEEE Transactions on Pattern Analysis and Machine
  Intelligence}, 40\penalty0 (12):\penalty0 2853--2867, 2017.

\bibitem[Xiang et~al.(2010)Xiang, Yang, Latecki, Liu, and Al]{2010Learning}
B.~Xiang, X.~Yang, L.~J. Latecki, W.~Liu, and E.~Al.
\newblock {Learning Context-Sensitive Shape Similarity by Graph Transduction}.
\newblock \emph{IEEE Transactions on Pattern Analysis and Machine
  Intelligence}, 32\penalty0 (5):\penalty0 861, 2010.

\bibitem[Xu et~al.(2019{\natexlab{a}})Xu, Luo, and Carin]{xu2019scalable}
Hongteng Xu, Dixin Luo, and Lawrence Carin.
\newblock Scalable gromov-wasserstein learning for graph partitioning and
  matching.
\newblock In \emph{Advances in Neural Information Processing Systems},
  volume~32, 2019{\natexlab{a}}.

\bibitem[Xu et~al.(2019{\natexlab{b}})Xu, Luo, Zha, and Duke]{xu2019gromov}
Hongteng Xu, Dixin Luo, Hongyuan Zha, and Lawrence~Carin Duke.
\newblock Gromov-wasserstein learning for graph matching and node embedding.
\newblock In \emph{International conference on machine learning}, pages
  6932--6941. PMLR, 2019{\natexlab{b}}.

\bibitem[Xu et~al.(2019{\natexlab{c}})Xu, Wang, Yu, Feng, Song, Wang, and
  Yu]{xu2019cross}
Kun Xu, Liwei Wang, Mo~Yu, Yansong Feng, Yan Song, Zhiguo Wang, and Dong Yu.
\newblock {Cross-lingual Knowledge Graph Alignment via Graph Matching Neural
  Network}.
\newblock In \emph{Proceedings of the 57th Annual Meeting of the Association
  for Computational Linguistics}, pages 3156--3161, 2019{\natexlab{c}}.

\bibitem[Zaslavskiy et~al.(2008)Zaslavskiy, Bach, and Vert]{zaslavskiy2008path}
Mikhail Zaslavskiy, Francis Bach, and Jean-Philippe Vert.
\newblock A path following algorithm for the graph matching problem.
\newblock \emph{IEEE Transactions on Pattern Analysis and Machine
  Intelligence}, 31\penalty0 (12):\penalty0 2227--2242, 2008.

\bibitem[Zass and Shashua(2006)]{zass2006doubly}
Ron Zass and Amnon Shashua.
\newblock Doubly stochastic normalization for spectral clustering.
\newblock \emph{Advances in Neural Information Processing Systems}, 19, 2006.

\end{thebibliography}
}

\appendix
\section{Lawler quadratic problem and iterative formule}

This section introduces the Lawler quadratic problem \cite{gold1996graduated,leordeanu2005spectral,leordeanu2009integer} and the corresponding constrained gradient method.

\subsection{Lawler quadratic problem}
The matching score of the node part can be defined by a measure function $\Omega_v( \cdot, \cdot)$:

\begin{equation}
\sum_{v_{i} \in V} \sum_{\tilde{v}_{\tilde{i}}\in \tilde{V}}  
  M_{i \tilde{i}} \Omega_v(F_i,\tilde{F}_{\tilde{i}}) = \textbf{m}^T\mathbf{k},
\label{sco.node}
\end{equation}
where $\mathbf{m} = \operatorname{vec}(M) \in \mathbb{R}^{n^2 \times 1}$ and $\mathbf{k} \in \mathbb{R}^{n^2 \times 1}$ vector used to store similarities between nodes' features.
Similarly, the total matching score of the edge part is defined as:
\begin{equation}
\frac{1}{2}\sum_{e_{ij} \in E} \sum_{\tilde{e}_{\tilde{i}\tilde{j}} \in \tilde{E}}
M_{i \tilde{i}} M_{j \tilde{j}} \Omega_e(A_{ij},\tilde{A}_{\tilde{i} \tilde{j}}) = \frac{1}{2} \mathbf{m}^{T} W \mathbf{m}.
\label{sco.edge}
\end{equation}
where $\Omega_e(A_{ij},\tilde{A}_{\tilde{i} \tilde{j}})$ represents the similarity between $E_{ij}$ and $\tilde{E}_{\tilde{i} \tilde{j}}$, and $W$ is an $n^2 \times n^2$ \textit{compatibility matrix} to store these similarities, i.e., $W_{i \tilde{i},j \tilde{j}} = \Omega_e(A_{ij},\tilde{A}_{\tilde{i} \tilde{j}})$. $M_{i \tilde{i}} M_{j \tilde{j}}=1$ means that $E_{ij}$ matches $\tilde{E}_{\tilde{i}\tilde{j}}$ (if both two edges exist). Combined two parts, the graph matching problem is formulated as
\begin{equation}
\begin{array}{cc}
\max\limits_{\mathbf{m} \in \Pi_{n^2 \times 1}} z(\mathbf{m}), \ \ z(\mathbf{m})=\frac{1}{2} \mathbf{m}^{T} W \mathbf{m}  + \lambda \mathbf{m}^T\mathbf{k}.

\label{eq.IQP}
\end{array}
\end{equation}
This problem is called Lawler quadratic problem \cite{1963The}. Since the computation on the $n^2 \times n^2$ compatibility matrix requires $O(n^4)$ operations, algorithms based on this formulation inevitably have space complexity $O(n^4)$ and time complexity $O(n^4)$.  When $W=A \otimes \tilde{A}$ and $\mathbf{k}= \operatorname{vec}(F \otimes \widetilde{F}^T)$, the problem \eqref{eq.IQP} can be rewritten as problem \eqref{eq.Object_fast}.  The derivation can be obtained directly from in \cite[Th.16.2.2]{Harville2008Matrix}. 
\subsection{Transition of the constrained gradient method}
The constrained gradient method to problem \eqref{eq.IQP} is
\begin{equation}
\begin{array}{cc}
    \mathbf{m}^{(t+1)}=(1- \alpha)\mathbf{m}^{(t)}  + \alpha \mathbf{d}^{(t)},\\
    \mathbf{d}^{(t)} =  \mathcal{P}(\nabla z(\mathbf{m}^{(t)}))= \mathcal{P}(W\mathbf{m}^{(t)} +\lambda \mathbf{k}).
    \end{array}
    \label{iter.quadratic}
\end{equation}
Computation of $W\mathbf{m}$ in each iteration requires $O(n^4)$ operations. According to \cite[Th.16.2.1]{Harville2008Matrix} that $ (A\otimes \tilde{A}) \mathbf{m} = \operatorname{vec}{(AM\tilde{A})}$, the iterative formula \eqref{iter.quadratic} can be rewritten as \eqref{iter.DSPFP}.

 \section{Common constraining operators}
The \textbf{Hungarian method} \cite{kuhn1955hungarian} can find a nearest permutation matrix $P$ to $X$:
\begin{equation}
\mathcal{P}_{H}(X)=\arg \min_{P\in \Pi_{n \times n}} \  \ \|P-X\|_{Fro}.
\label{eq:linear assign}
\end{equation}
This method can also transform a solution of \eqref{eq.Object_relaxedKB} into a matching matrix. This algorithm has $O(n^3)$ time complexity. Since the Hungarian method only find one solution of the linear assignment problem \eqref{eq:linear assign}, other solutions are ignored when multiple solutions exist.

\textbf{Greedy linear method} is an approximating method for the linear assignment problem \eqref{eq:linear assign}. By comparison, the Hungarian method is designed to find a global solution for the linear assignment problem, so it inevitably discards some high-value candidates and retains some low-value ones. To prioritize retaining high-valued candidates, \citet{leordeanu2005spectral} propose a greedy linear method, which does not guarantee the optimum. For a given input matrix $X$, this method works as follows:

 1. Initialize an $n \times n$ zero-valued matrix $P$.

 2. Find the index $(i, j)$ such that $X(i, j)$ is the largest element of $X$. 

 3. Let $P(i, j)= 1$, $X(i, :) =0$ and $X(:, j) =0$.

 4. Repeat step 2 and 3 until $X$ is empty, and return $P$.

\noindent Although this greedy method has $O(n^3)$ time complexity, it is relatively cheaper and easier than Hungarian  \cite{leordeanu2005spectral}. 

\textbf{Alternating projection method} \cite{zass2006doubly} aims to find a nearest doubly stochastic matrix $P$ to $X$:
\begin{equation}
\mathcal{P}_{D}(X) = \arg \min_{P\in \Sigma_{n \times n}} \|P-X\|_{Fro}.
\label{eq:doubly assign2}
\end{equation}
The solution of \eqref{eq:doubly assign2} is a convex combination of all optimal solutions of \eqref{eq:linear assign} \cite{peyre2017computational}, so $\mathcal{P}_{D}(X)$ can contain information of multiple discrete solutions of \eqref{eq:linear assign}. This method works as follows:
\begin{equation}
\mathcal{P}_{1}(X)=X+(\frac{\mathbf{1}}{n}I+\frac{\mathbf{1}^{T}X\mathbf{1}}{n^2}I-\frac{\mathbf{1}}{n}X)\mathbf{1}\mathbf{1}^{T}-\frac{\mathbf{1}}{n}\mathbf{1}\mathbf{1}^{T}X,
\end{equation}
\begin{equation}
\mathcal{P}_{2}(X)=\frac{X+|X|}{2},
\end{equation}
\begin{equation}
\mathcal{P}_{D}(X)=\dots \mathcal{P}_{2}(\mathcal{P}_{1}(\mathcal{P}_{2}(\mathcal{P}_{1}(\mathcal{P}_{2}(\mathcal{P}_{1}(X)))))).
\end{equation} 
Each iteration requires $O(n^2)$ operations. Since this projection is nonexpansive, the number of iterations positively correlates with the magnitude of $X$ \cite{zass2006doubly}. To avoid potential expensive cost, \citet{lu2016fast} sets the maximum number of iterations to 30. However, limited iterations cannot guarantee that the $\mathcal{P}_{D}(X)$ is doubly stochastic (more details can be found in the experiment part).



\end{document}